\documentclass[11pt]{amsart}
\usepackage{url}
\usepackage{amssymb,amscd}
\usepackage[all]{xy}

 \usepackage{mathrsfs,amsfonts,amsmath,amsthm,amssymb}
\usepackage{graphics}
\usepackage[all]{xy}
\xyoption{curve}
\xyoption{import}
\xyoption{arc}
\xyoption{ps}
\usepackage{amsmath}
\usepackage{graphicx}

  \usepackage[a4paper, margin=2.5cm]{geometry}

 \usepackage{times}

\begin{document}

\numberwithin{equation}{section}

{\theoremstyle{definition}\newtheorem{definition}{Definition}[section]
\newtheorem{notation}[definition]{Notation}
\newtheorem{remnot}[definition]{Remarks and notation}
\newtheorem{terminology}[definition]{Terminology}
\newtheorem{remark}[definition]{Remark}
\newtheorem{remarks}[definition]{Remarks}
\newtheorem{example}[definition]{Example}
\newtheorem{examples}[definition]{Examples}
\newtheorem{deflem}[definition]{Definition-Lemma}
\newtheorem{proposition}[definition]{Proposition}
\newtheorem{lemma}[definition]{Lemma}
\newtheorem{theorem}[definition]{Theorem}
\newtheorem{corollary}[definition]{Corollary}

\newcommand{\ci}{C^{\infty}}
\newcommand{\A}{\mathscr{A}}
\newcommand{\Cat}{\mathscr{C}}
\newcommand{\Dnc}{\mathscr{D}}
\newcommand{\E}{\mathscr{E}}
\newcommand{\F}{\mathscr{F}}
\newcommand{\T}{\mathscr{T}}
\newcommand{\gr}{\mathscr{G}}
\newcommand{\go}{\mathscr{G} ^{(0)}}
\newcommand{\hr}{\mathscr{H}}
\newcommand{\ho}{\mathscr{H} ^{(0)}}
\newcommand{\rgr}{\mathscr{R}}
\newcommand{\rgo}{\mathscr{R} ^{(0)}}
\newcommand{\lr}{\mathscr{L}}
\newcommand{\lo}{\mathscr{} ^{(0)}}
\newcommand{\gd}{\mathscr{G}^{\mathbb{R}^2}}
\newcommand{\gt}{\mathscr{G} ^{T}}
\newcommand{\I}{\mathscr{I}}
\newcommand{\Nb}{\mathscr{N}}
\newcommand{\Kom}{\mathscr{K}}
\newcommand{\ops}{\mathscr{O}}
\newcommand{\Pb}{\mathscr{P}}
\newcommand{\sw}{\mathscr{S}}
\newcommand{\Uo}{\mathscr{U}}
\newcommand{\Vo}{\mathscr{V}}
\newcommand{\Wo}{\mathscr{W}}
\newcommand{\Rr}{\mathbb{R}}
\newcommand{\Nat}{\mathbb{N}}
\newcommand{\C}{\mathbb{C}}
\newcommand{\src}{\mathscr{S}_{c}}
\newcommand{\cc}{C_{c}^{\infty}}
\newcommand{\cg}{C_{c}^{\infty}(\gr)}
\newcommand{\cgo}{C_{c}^{\infty}(\go)}
\newcommand{\ct}{C_{c}^{\infty}(\gr^T)}
\newcommand{\ca}{C_{c}^{\infty}(A\gr)}
\newcommand{\Un}{{U}^{(n)}}
\newcommand{\Du}{D_{\mathscr{U}}}

\def\to{\longrightarrow}

\def\L{\mathop{\wedge}}

  \renewcommand{\a}{\sigma}
  \renewcommand{\b}{\beta}
   \newcommand{\io}{\iota}
  
  \newcommand{\eps}{\epsilon}
  \renewcommand{\d}{\delta}
  \newcommand{\pa}{\partial}

  \newcommand{\ind}{{\bf  Index }}

  \newcommand{\CC}{{\mathbb C}}
  \newcommand{\RR}{{\mathbb R}}
  \newcommand{\ZZ}{{\mathbb Z}}
  \newcommand{\QQ}{{\mathbb Q}}
  \newcommand{\NN}{{\mathbb  N}}
    \newcommand{\KK}{{\mathbb  K}}

  \newcommand{\PP}{{\mathbb P}}

\def\gpd{\,\lower1pt\hbox{$\longrightarrow$}\hskip-.24in\raise2pt
             \hbox{$\longrightarrow$}\,}

\title{Twisted longitudinal  index theorem for foliations and wrong way functoriality}

  \author{ Paulo Carrillo Rouse}   
  \address{Institut de Math\'ematiques de Toulouse \\
Universit\'e de Toulouse III\\  118, route de Narbonne \\
 31400 Toulouse \\ France} 
 \email{paulo.carrillo@math.univ-toulouse.fr} 

  \author{Bai-Ling Wang}
  \address{Department of Mathematics\\
  Australian National University\\
  Canberra ACT 0200 \\
  Australia}
  \email{wangb@maths.anu.edu.au}

\bigskip
\everymath={\displaystyle}

\begin{abstract}
\noindent 
For a Lie groupoid $\gr$   with a twisting $\sigma$ (a $PU(H)$-principal bundle over $\gr$), we use the (geometric) deformation quantization techniques supplied by Connes tangent groupoids to define an analytic index morphism 
\[
\xymatrix{
K^{*}(A^*\gr, \pi^*\sigma )\ar[rr]^{\ind_{a,(\gr,\sigma)}} & & K ^{*}(\gr, \a) }
\]
in twisted K-theory. In the case the twisting is trivial we recover the analytic index morphism of the groupoid.

For a smooth foliated manifold with  twistings on the holonomy groupoid we prove the twisted analog of Connes-Skandalis longitudinal index theorem. When the foliation is given by fibers of a fibration, our index coincides with the one recently introduced  by Mathai-Melrose-Singer.

We construct the pushforward map in twisted K-theory associated to any smooth (generalized) map 
$f:W\longrightarrow M/F$ and a twisting $\sigma$ on the holonomy groupoid $M/F$, next we use the longitudinal index theorem to prove the functoriality of this construction. We generalize in this way the wrong way functoriality results of Connes-Skandalis when the twisting is trivial and of Carey-Wang  for manifolds.
\end{abstract}

\maketitle

\tableofcontents

\section*{Introduction}
This paper is part of a longstanding project where we aim to study and develop an index theory (\'a l'Atiyah-Singer and Connes-Skandalis) for foliations under the presence of twistings. 

The Atiyah-Singer index theorem in  \cite{AS},  which states that 
for an elliptic pseudo-differential operator on a compact manifold, the analytical index is equal to the topological index, has been generalized  in many other situations.  In  particular,  Connes and Skandalis established  an index theorem for longitudinal elliptic operators on foliated manifolds, \cite{CS}.   

For a smooth  foliated manifold $(M, F)$, there is a canonically defined 
$C^*$-algebra $C^*(M, F) = C^*(\gr_{M })$  which is the $C^*$-algebra of the underlying 
holonomy groupoid $\gr_{M}=M/F$.  An
elliptic pseudo-differential operator $D$  along the leaves of the foliation defines an element
\[
\ind_{a, M/F} (D) \in K_0(C^*(M, F)),
\]
in the K-theory of the $C^*$-algebra. This is done by using  pseudo-differential calculus for the holonomy groupoid $\gr_{M }$.  The element $\ind_a(D)$  is called the analytical index of $D$, which depends only on the longitudinal symbol of $D$.  Hence, one has the analytical index as a
group homomorphism  associated to the foliated manifold $(M, F)$
\[
\ind_{a, M/F} : K^0(F^*) \longrightarrow K_0(C^*(M, F)).
\]
The topological index for $(M, F)$ is obtained by applying the  Thom isomorphism and Bott periodicity  in topological K-theory to an auxiliary embedding 
\[
\io: M \longrightarrow \RR^{2n}.
\]
Let $N$ be the total space of the normal bundle to the leaves.  The Thom isomorphism implies
that 
\[
K^0(F^*) \cong K^0(N).
\]
 There is an open
neighborhood of the 0-section in $N$ which is homeomorphic to an open transversal $T$
of  the induced foliation $\tilde{F}$ on  $\tilde{M} = M\times \RR^{2n}$.  For a suitable open neighborhood   $\Omega$  of $T $ in $M\times \RR^{2n}$,   the $C^*$-algebra 
$C^*(\Omega, \tilde{F}|_\Omega)$ is Morita equivalent to $C_0(T)$ (the $C^*$-algebra of continuous functions on $T$  vanishing at infinity). So one has 
\[
K^0(N)\cong K^0(T) \cong K_0(C^*(\Omega, \tilde{F}|_\Omega)) \to K_0(C^*(\tilde{M}, \tilde{F}))
\]
where the last map is induced by the inclusion $C^*(\Omega, \tilde{F}|_\Omega) \subset 
C^*(\tilde{M}, \tilde{F})$. The Bott periodicity 
gives rise to
\[
K_0(C^*(\tilde{M}, \tilde{F})) = K_0(C^*(M, F) \otimes C_0(\RR^{2n}) ) \cong  
K_0(C^*(M, F)).
\]
This leads to the topological index for a  foliated manifold $(M, F)$
\[
\ind_{t, M/F} : K^0(F^*) \longrightarrow K_0(C^*(M, F)).
\]
In \cite{CS}, Connes and Skandalis showed 
that 
\[
\ind_{a, M/F} = \ind_{t, M/F}.
\]

There is an equivalent definition (\cite{HS}, \cite{Concg}, \cite{MP}) of
the analytical index   for a foliated manifold $(M, F)$ using  the Connes tangent groupoid
\[
\gr_{M }^T:=F\times \{0\} \bigsqcup \gr_{M}\times (0,1]\rightrightarrows M\times [0,1]
\]
of the holonomy groupoid $\gr_{M }$.   The definition works for any Lie groupoid  $\gr$. Let $A\gr$ be the associated Lie algebroid which can be viewed as a Lie groupoid given by 
the vector bundle structure, then one has 
the following exact sequence (\cite{HS})
\[
0 \rightarrow C^*(\gr \times (0,1]) \longrightarrow C^*(\gr^T)
\stackrel{ev_0}{\longrightarrow}
C^*(A\gr) \rightarrow 0.
\]
Monthubert-Pierrot in \cite{MP} show that
\[
\ind_{a, M/F} =(ev_1)_*\circ
(ev_0)_*^{-1} \circ \F := \ind_{a, \gr_{M }}:  K^0(F^*) \to  K_0(C^*(M, F))
\]
where $\F: K^0(F^*) = K^0(A^*\gr) \cong K_0(C^*(A\gr)) $ is the isomorphism 
induced by the  Fourier isomorphism $C^*(A\gr) \cong C_0(A^*\gr)$ and the morphisms $ev_t$ are the respective evaluations. 

 There is also an equivalent definition of the topological index  for a  foliated manifold $(M, F)$ (see \cite{Ca4}) by realizing the Thom inverse morphism of a complex vector bundle as the deformation index of some groupoid (The Thom groupoid, \cite{DLN} Theorem 6.2).  The resulting topological index is denoted by
 \[
  \ind_{t, \gr_{M }}:  K^0(F^*) \to  K_0(C^*(M, F))
\]
One has the following equivalences between these four index morphisms:
\[
\ind_{a, \gr_{M }} = \ind_{a, M/F}  = \ind_{t, M/F}   =   \ind_{t, \gr_{M }}.
\]

The approach taken in this paper is of geometrical nature, we are going to define our indices and propose the corresponding index theorem by using deformation groupoid techniques. We will leave the discussion of the pseudodifferential calculus behind  these for later works.

As we mentioned, the analytic index morphism associated to a Lie groupoid arises naturally from a geometric construction, that of the Connes tangent groupoid, \cite{HS,Concg,MP,NWX}. This groupoid encodes the deformation of the groupoid to the Lie algebroid. The use of deformation groupoids in index theory has proven to be very useful to prove index theorems, as well as to establish index theories in many different settings, \cite{HS,Concg,DLN,ANS,Ca4,Hig08,Ca5,CMont,vErp}. For example, even in the classic case of Atiyah-Singer theorem, Connes gives a very simple conceptual proof using the tangent groupoid (\cite{Concg}, II.5).

In this paper, we use these strict deformation quantization techniques to work out an index theory for foliation groupoids
in the presence of a twisting which generalizes the index theory \cite{W08}  for
a  smooth manifold with a twisting 
$\a: M \to K(\ZZ, 3)$, and Connes-Skandalis index
theory when the twisting is trivial. 

Following the construction of $KK$-elements developped in \cite{HS} (in particular section III) by Hilsum-Skandalis, we construct an analytic index morphism of any Lie groupoid $\gr\rightrightarrows M$ with a twisting given by a Hilsum-Skandalis' generalized  morphism $\xymatrix{\a: \gr \ar@{-->}[r] & PU(H)}$  
\[
\xymatrix{
K^{*}(A\gr, \pi^*\sigma )\ar[rr]^{\ind_{a, (\gr, \sigma)}} & & K ^{*}(\gr, \a). }
\]
This index  $\ind_{a, (\gr, \sigma)}$ takes values in the twisted K-theory of $(\gr, \a)$ (see \cite{TXL} for twisted K-theory for groupoids or definition \ref{twistedkth} below). When the twisting is trivial we recover the analytic index morphism of the groupoid defined by Monthubert and Pierrot in \cite{MP}.

Equivalently a twisting is given by a $PU(H)$-principal bundle over the units of the groupoid $\gr$ together with a compatible $\gr$-action. When the groupoid is a manifold the twistings we are considering here are equivalently given by Dixmier-Douady classes and the twisted K-theory is the   twisted K-theory for manifolds, \cite{ASeg,Kar08}. In general, our setting also includes, as particular cases, equivariant twisted K-theory and twisted K-theory for orbifolds, \cite{TXL}. 

For general Lie groupoids,  we  prove that our twisted index satisfies the following three properties (Propositions \ref{Botttwist}, \ref{opentwist} and \ref{Thomtwist}):
\begin{itemize}
\item[(i)] It is compatible with the Bott morphism,
$\it{i.e.}$, the following diagram is commutative
\[
\xymatrix{
K^*(A^*\gr,\pi^*\sigma_0) \ar[rrr]^-{\ind_{a,(\gr,\sigma)}} \ar[d]_-{Bott}
& & & K^*(\gr,\sigma) \ar[d]^-{Bott} \\
K^*(A^*\gr\times\Rr^2,\pi^*(\sigma\circ p)_0) \ar[rrr]_-{\ind_{a,(\gr\times \Rr^2,\sigma\circ p)}} & & & K^*(\gr\times \Rr^2,\sigma\circ p)
}
\]
where $p:\gr\times \Rr^2\longrightarrow \gr$ is the projection,  $\a_0$ and $(\a\circ p)_0$ are the induced twistings on the unit spaces of $\gr$ and $\gr\times \Rr^2$ respectively. 
\item[(ii)] Let $\hr\stackrel{j}{\hookrightarrow}\gr$ be an inclusion as an open subgroupoid.
The following diagram is commutative:
\[
\xymatrix{
K^*(A^*\hr,\pi^*(\sigma\circ j)_0) \ar[rr]^-{\ind_{a,(\hr,\sigma\circ j)}} \ar[d]_{j_!}
& & K^*(\hr,\sigma\circ j) \ar[d]^{j_!} \\
K^*(A^*\gr,\pi^*\sigma_0) \ar[rr]_{\ind_{a,(\gr,\sigma)}} & & K^*(\gr,\sigma).
}
\]
where the vertical maps are induced from the inclusions by open subgroupoids.
\item[(iii)] Let $N\longrightarrow T$ be a real vector bundle. Consider the product groupoid
$N\times_TN\rightrightarrows N$ This groupoid has a  Lie algebroid 
$N\oplus N\stackrel{\pi_N}{\longrightarrow} N$.  The groupoid $N\times_TN\rightrightarrows N$ is Morita equivalent to the identity groupoid $T\rightrightarrows T$. That is,  there is an isomorphism in the Hilsum-Skandalis category 
$$N\times_T N\stackrel{\mathscr{M}}{\longrightarrow}T.$$

In the presence of a twisting $\beta$ on the space 
$T$ the twisted index 
\[\xymatrix{
K^*(N\oplus N,\pi^*\beta)\ar[rrr]^{\ind_{a,(N\times_TN, \beta \circ \mathscr{M})}}&&&
K^*(N\times_TN, \beta \circ \mathscr{M})   \ar[r]_{\qquad \cong }
&
K^*(T,\beta)}\]
is the inverse (modulo a Morita equivalence and a Fourier isomorphism) of the Thom isomorphism (\cite{CW} \cite{Kar08}) in twisted K-theory $K^*(T,\beta)\stackrel{T_{\beta}}{\longrightarrow}K^*(N\oplus N,\pi_T^*\beta)$.

\end{itemize}

These properties are the analogs to the axioms stated in \cite{AS} and will allow us to prove a twisted index theorem for foliations and in general we can then study   twisted index theory for these objects. Indeed, for the case of regular foliations,  we extend the topological index of Connes-Skandalis (Definition \ref{twistedtopindex} below) to the twisted case and prove the following  twisted longitudinal index theorem.

\begin{theorem}
For a regular foliation $(M,F)$ with a twisting $\xymatrix{\a:  M/F  \ar@{-->}[r] & PU(H)}$ on the space of leaves,  we have the following equality of morphisms:
\begin{equation}\nonumber
\ind_{a,(M/F,\sigma)}=\ind_{t,(M/F,\sigma)}: K^*(F, \pi^*\a) \longrightarrow 
K^*(M/F, \a)
\end{equation}
\end{theorem}

In fact, any index morphism for foliations with twistings satisfying the three properties  (i), (ii), (iii)  is equal to the twisted topological index by exactly the same proof.  There is a second definition of the twisted analytic index using projective pseudo-differential operators along the leaves which also satisfies the three properties.  We will return to this issue in a separate paper.

When the foliation is  given by a fibration $\pi: M\to B$  and the twisting is a torsion class
in $H^3(B, \ZZ)$, our index coincides with the one recently introduced in \cite{MMSI} by Mathai-Melrose-Singer since the topological indices of both papers are the same. We remark that they also give the correspondant cohomological formula which was also recently proved by supperconnection methods by Benameur-Gorokhovsky \cite{BG}. We will leave the study of Chern-Connes type formulas in the twisted foliation context for further works.

In this paper, we also construct, for every smooth (generalized) map $f:W\longrightarrow M/F$ and a twisting $\sigma$ on $M/F$, a pushforward morphism in twisted K-theory
\[
f_!:   K^{*+i}(W,  f^*\a_0 + o_{TW \oplus  f^*\nu_F})\longrightarrow K^i(M/F,\sigma),
\]
for an induced twisting on $W$. 
We then use the twisted  longitudinal index theorem to prove the functoriality of this construction. 

\begin{theorem}
The pushforward morphism is   functorial,  that is, if we have a composition of smooth maps
\begin{equation}
Z\stackrel{g}{\longrightarrow}W\stackrel{f}{\longrightarrow}M/F,
\end{equation}
and a twisting $\sigma:M/F\longrightarrow PU(H)$ then $(f\circ g)_! =f_! \circ g_! $
\end{theorem}

When $\sigma$ is trivial and $f: W\to M/F$ is K-oriented (that is  $TW\oplus \nu_F$ is K-oriented), our push-forward map  in Definition \ref{Main:def}  agrees with the one constructed in \cite{CS}. Also, when the foliation consists on a single leaf (hence a manifold) but the twisting is not necessarily trivial, we obtain otherwise the push-forward map defined in \cite{CW} by Carey-Wang.

We remark an important difference with respect to Connes-Skandalis approach in \cite{CS}. Indeed, we prove first, in a geometric way, the longitudinal index theorem and then we use it to prove the pushward functoriality. Recall that in \cite{CS} it is done conversely with analytic methods. Indeed, their fundamental technical result (as remarked by Connes-Skandalis), Lemma 4.7 in loc.cit., is proved by means of the longitudinal pseudodifferential calculus and the KK-elements associated to it. With this result they prove the functoriality of the pushforward and then the longitudinal index theorem. In this paper, the twisted analog of the lemma mentioned above (Proposition \ref{CS47}) is proved using the longitudinal index theorem. The strict deformation quantization approach to index theory for foliations (without twistings) have already used in this sense in previous works by the first author in \cite{Ca4,Ca5}.

 Some of the results presented in this paper were announced in \cite{CaWangCRAS}. There is however an important difference in the way of proving the pushforward functoriality. The arguments stated in that note for proving such a result were independent of the longitudinal index theorem. It is indeed possible to prove it directly by deforming some KK-elements, in the spirit of Theorem 6.2 in \cite{DLN} or Theorem \ref{Thomtwist} below. 
 
\bigskip

{\bf Acknowledgments.} We would like to thank two institutions for their hospitality while part of this work was being done: Max Planck Institut for Mathematics (Carrillo Rouse and Wang) and ANU (Carrillo Rouse). We would also like to thank Georges Skandalis for very useful discussions and comments. The first author would like to thank Jean Renault for enlightening conversations about groupoids, extensions and index theory that were very stimulating for this paper.

\section{Twistings  on Lie groupoids}

In this section, we review the notion of twistings on Lie groupoids and 
discuss some examples which appear in this paper.
Let us recall what a groupoid is:

\begin{definition}
A $\it{groupoid}$ consists of the following data:
two sets $\gr$ and $\go$, and maps
\begin{itemize}
\item[(1)]  $s,r:\gr \rightarrow \go$ 
called the source map and target map respectively,
\item[(2)]  $m:\gr^{(2)}\rightarrow \gr$ called the product map 
(where $\gr^{(2)}=\{ (\gamma,\eta)\in \gr \times \gr : s(\gamma)=r(\eta)\}$),
\end{itemize}
together with  two additional  maps, $u:\go \rightarrow \gr$ (the unit map) and 
$i:\gr \rightarrow \gr$ (the inverse map),
such that, if we denote $m(\gamma,\eta)=\gamma \cdot \eta$, $u(x)=x$ and 
$i(\gamma)=\gamma^{-1}$, we have 
\begin{itemize}
\item[(i).]$r(\gamma \cdot \eta) =r(\gamma)$ and $s(\gamma \cdot \eta) =s(\eta)$.
\item[(ii).]$\gamma \cdot (\eta \cdot \delta)=(\gamma \cdot \eta )\cdot \delta$, 
$\forall \gamma,\eta,\delta \in \gr$ whenever this makes sense.
\item[(iii).]$\gamma \cdot x = \gamma$ and $x\cdot \eta =\eta$, $\forall
  \gamma,\eta \in \gr$ with $s(\gamma)=x$ and $r(\eta)=x$.
\item[(iv).]$\gamma \cdot \gamma^{-1} =u(r(\gamma))$ and 
$\gamma^{-1} \cdot \gamma =u(s(\gamma))$, $\forall \gamma \in \gr$.
\end{itemize}
For simplicity, we denote a groupoid by $\gr \rightrightarrows \go $. A strict morphism $f$ from
a  groupoid   $\hr \rightrightarrows \ho $  to a groupoid   $\gr \rightrightarrows \go $ is  given
by  maps in 
\[
\xymatrix{
\hr \ar@<.5ex>[d]\ar@<-.5ex>[d] \ar[r]^f& \gr \ar@<.5ex>[d]\ar@<-.5ex>[d]\\
\ho\ar[r]^f&\go
}
\]
which preserve the groupoid structure, i.e.,  $f$ commutes with the source, target, unit, inverse  maps, and respects the groupoid product  in the sense that $f(h_1\cdot h_2) = f (h_1) \cdot f(h_2)$ for any $(h_1, h_2) \in \hr^{(2)}$.

\end{definition}

In  this paper we will only deal with Lie groupoids, that is, 
a groupoid in which $\gr$ and $\go$ are smooth manifolds, and $s,r,m,u$ are smooth maps (with s and r submersions, see \cite{Mac,Pat}). For two subsets $U$ and $V$ of $\go$,
 we use the notation
$\gr_{U}^{V}$ for the subset 
\[
\{ \gamma \in \gr : s(\gamma) \in U,\, 
r(\gamma)\in V\} .
\]

Lie groupoids generalize the notion of Lie groups. For Lie groupoids there is also a notion of Lie algebroid  playing the role of the Lie algebra in Lie theory. 

\begin{definition}[The Lie algebroid of a Lie groupoid]
Let $\gr \rightrightarrows  \go$ be a Lie groupoid. Denote by $A\gr$ the normal bundle associated to the inclusion $\go \subset \gr$. We refer to this vector bundle 
$\pi: A\gr \rightarrow \go$ as the Lie algebroid of $\gr$.
\end{definition} 

\begin{remark}
The Lie algebroid $A\gr$  of $\gr$ has a Lie algebroid structure. For see this it is usually more convenient to identify $A\gr$ with the restriction to the unit space of the vertical tangent bundle along the fiber of the source map $s: \gr\to \go$. This identification is not canonical, it corresponds indeed to a particular choice of splitting of the short exact sequence of vector bundles over $\go$.
$$0\rightarrow T\go\rightarrow T_{\go}\gr\rightarrow A\gr \rightarrow 0.$$
The Lie algebroid structure consists of
\[
(\pi: A\gr \rightarrow \go, dr: A\gr \to T\go, [\  ,\  ]_{A\gr})
\]
where the anchor map $dr$  is the restriction of the differential of $r$ to $A\gr$, and
the Lie bracket  $[\  ,\  ]_{A\gr}$ on the space of section of $A\gr$ is given the usual Lie bracket of the right 
invariant vector fields, restricted to $\go$. We will not make use on this paper of this additional structure.
\end{remark}

\subsection{Generalized morphisms: the Hilsum-Skandalis category}
 
Lie groupoids form a category with  strict  morphisms of groupoids. It is now classical in Lie groupoid's theory that the right category to consider is the one in which Morita equivalences correspond precisely to isomorphisms.  We review some basic definitions and properties of generalized morphisms between Lie groupoids, see \cite{TXL} section 2.1, or 
\cite{HS,Mr,MM} for more detailed discussions.

We first recall   the notions of  groupoid actions and principal bundles over groupoids. Given a Lie groupoid $\gr \rightrightarrows \go$,  
a right  $\gr$-bundle over a manifold $M$ is  a manifold $P$ equipped with maps 
 as in 
\[
\xymatrix{
 P \ar[d]_\pi   \ar[rd]_-\epsilon&\gr \ar@<.5ex>[d]_{r\ } \ar@<-.5ex>[d]^{\ s}  \\
M &\go, 
}
\]
together  with  a smooth right  $\gr$-action $\mu: P \times_{(\epsilon, r )} \gr \to P$, $\mu(p, \gamma) = p \gamma$,   such that  $\pi( p \gamma)  = \pi(p)$ and 
$p( \gamma_1\cdot \gamma_2) =   (p  \gamma_1)  \gamma_2 $ for any 
$(\gamma_1, \gamma_2) \in \gr^{(2)}$.  Here  $P \times_{(\epsilon, r )} \gr$  denotes  the fiber
product of $\epsilon: P\to \go$ and $r: \gr\to \go$. A left $\gr$-bundle can be defined similarly.
 A  $\gr$-bundle $P$ is called 
principal if
\begin{enumerate}
\item[(i)] $\pi$ is a surjective submersion, and
\item[(ii)]  the map  $P\times_{(\epsilon, r)} \gr \to P\times_M P$,  $(p, \gamma) \mapsto (p,  p\gamma)$  
is a diffeomorphism. 
\end{enumerate}
A principal $\gr$-bundle $P$  over $M$ is called  locally trivial if $P$ is isomorphic to 
a  principal $\gr$-bundle
\[
\dfrac{\bigsqcup_i \Omega_i \times \gr } {\{ (x, j, \gamma )  \sim (x,i, g_{ij}(x) \cdot \gamma)\}}
\]
defined by
 a $\gr$-valued  1-cocycle on $M$
\[
\{
g_{ij}:    \Omega_i \cap \Omega_j  \longrightarrow \gr| g_{ij}(x) \cdot g_{jk}(x) =  g_{ij} (x) \}
\]
with respect to a cover $\{\Omega_i\}$ of $M$. 



  Let $\gr \rightrightarrows \go$ and  
$\hr \rightrightarrows \ho$ be two Lie groupoids.  A principal $\gr$-bundle over $\hr$ 
is a right  principal $\gr$-bundle over $\ho$
which is also a left $\hr$-bundle over $\go$ such that the   the right $\gr$-action and the left 
$\hr$-action commute,  formally denoted by
\[
\xymatrix{
\hr \ar@<.5ex>[d]\ar@<-.5ex>[d]&P_f \ar@{->>}[ld] \ar[rd]&\gr \ar@<.5ex>[d]\ar@<-.5ex>[d]\\
\ho&&\go.
}
\]
A  $\gr$-valued 1-cocycle on  $\hr$ with respect to  an indexed open covering $\{\Omega_i\}_{i\in I}$ of $\ho$ is  a collection of smooth maps 
 $$f_{ij}:\hr_{\Omega_j}^{\Omega_i} \to\gr,$$
 satisfying the following cocycle condition:
$\forall \gamma \in \hr_{ij}$ and $\forall \gamma'\in \hr_{jk}$ with $s(\gamma)=r(\gamma')$, we have
\begin{center}
$f_{ij}(\gamma)^{-1}=f_{ji}(\gamma^{-1})$ and $f_{ij}(\gamma)\cdot f_{jk}(\gamma')=f_{ik}(\gamma\cdot \gamma').$
\end{center}
We will denote this data by $f=\{(\Omega_i,f_{ij})\}_{i\in I}$.  It is equivalent to  a 
 strict morphism of groupoids 
\begin{equation}\nonumber
\xymatrix{
\hr_{\Omega}=\bigsqcup_{i,j}\hr_{\Omega_j}^{\Omega_i} \ar@<.5ex>[d]\ar@<-.5ex>[d]\ar[rr]^-f &&\gr\ar@<.5ex>[d]\ar@<-.5ex>[d]\\
 \bigsqcup_{i}\Omega_{i}\ar[rr]&&\go.
}
\end{equation}
 Associated to  a $\gr$-valued 1-cocycle on  $\hr$, there is a canonical defined  principal $\gr$-bundle over $\hr$.  In fact, any principal $\gr$-bundle over $\hr$ is locally trivial (Cf. \cite{MM}).  
 


 
  
   We can now define generalized morphisms between two Lie groupoids.

\begin{definition}[Generalized morphism]\label{HSmorphisms}
Let $\gr \rightrightarrows \go$ and 
$\hr \rightrightarrows \ho$ be two Lie groupoids. 
A   right $\gr-$principal bundle over $\hr$, also called a generalized morphism (or Hilsum-Skandalis morphism),   
$f:  \xymatrix{\hr \ar@{-->}[r] &  \gr}$, is given by one of the three equivalent data:
\begin{enumerate}
\item A  locally trivial  right  principal $\gr$-bundle $P_f$  over  $\hr$ 
\[
\xymatrix{
\hr \ar@<.5ex>[d]\ar@<-.5ex>[d]&P_f \ar@{->>}[ld] \ar[rd]&\gr \ar@<.5ex>[d]\ar@<-.5ex>[d]\\
\ho&&\go.
}
\]

\item A 1-cocycle $f=\{(\Omega_i,f_{ij})\}_{i\in I}$ on $\hr$ with values in $\gr$.

\item A 
 strict morphism of groupoids 
\begin{equation}\nonumber
\xymatrix{
\hr_{\Omega}=\bigsqcup_{i,j}\hr_{\Omega_j}^{\Omega_i} \ar@<.5ex>[d]\ar@<-.5ex>[d]\ar[rr]^-f &&\gr\ar@<.5ex>[d]\ar@<-.5ex>[d]\\
 \bigsqcup_{i}\Omega_{i}\ar[rr]&&\go,
}
\end{equation}
for an open cover  $\Omega= \{\Omega_i\}$ of $\ho$.
 \end{enumerate}
 Two generalized morphisms $f$ and $g$ are called equivalent if the corresponding
  right $\gr$-principal bundles $P_f$ and $P_g$ over $\hr$ are isomorphic.
\end{definition}

As the name suggests,  generalized morphism  generalizes the notion of strict morphisms and can be composed. Indeed, if $P$ and $P'$ are generalized morphisms from $\hr$ to $\gr$ and from $\gr$ to $\lr$ respectively, then 
$$P\times_{\gr}P':=P\times_{\go}P'/(p,p')\sim (p\cdot \gamma, \gamma^{-1}\cdot p')$$
is a generalized morphism from $\hr$ to $\lr$.  Consider the category $Grpd_{HS}$ with objects Lie groupoids and morphisms given by equivalent classes of generalized morphisms. There is a functor
\begin{equation}\label{grpdhs}
Grpd \longrightarrow Grpd_{HS}
\end{equation}
where $Grpd$ is the strict category of groupoids. Then the composition is associative   in
$Grpd_{HS}$. 
Two groupoids are Morita equivalent if they are isomorphic in $Grpd_{HS}$.
  For example, 
given a Lie groupoid $\hr\rightrightarrows\ho$ and an open covering $\{\Omega_i\}_i$ of $\ho$, the canonical strict morphism of groupoids $\hr_{\Omega}\longrightarrow \hr$ is a Morita equivalence. 

 \subsection{Twistings on  Lie groupoids}
 
 In this paper,  we are only going to consider $PU(H)$-twistings on Lie groupoids 
where $H$ is an infinite dimensional, complex and separable
Hilbert space, and $PU(H)$ is the projective unitary group $PU(H)$  with the topology induced by the
norm topology on $U(H)$. 

\begin{definition}\label{twistedgroupoid}
A  twisting $\sigma$  on a   Lie  groupoid $\gr \rightrightarrows \go$  is given by  a generalized morphism 
\[ \xymatrix{
\a: \gr \ar@{-->}[r]  & PU(H).}
\]
Here $PU(H)$ is viewed  as a Lie groupoid with the unit space $\{e\}$. Two twistings 
$\a$ and $\a'$ are called equivalent if they are  equivalent as generalized morphisms.
\end{definition}
 
 So a twisting on a Lie groupoid $\gr$ is  a locally trivial  right  principal $PU(H)$-bundle over $\gr$
\begin{equation}\label{gpalpha}
\xymatrix{
\gr \ar@<.5ex>[d]\ar@<-.5ex>[d]&P_{\sigma} \ar[ld] \ar[rd]&PU(H) \ar@<.5ex>[d]\ar@<-.5ex>[d]\\
\go&&\{e\}.
}
\end{equation}
 Equivalently,   a twisting on  $\gr$ is given by
 a $PU(H)$-valued 1-cocycle on $\gr$
\[ 
g_{ij}:   \gr_{\Omega_j}^{\Omega^i} \longrightarrow PU(H)
\]
for an open cover $\Omega= \{\Omega_i\}$ of $\go$. That is,  a twisting $\sigma$ on a Lie  groupoid $ \gr $ is given by a strict morphism of groupoids 
\begin{equation}\label{galpha}
\xymatrix{
\gr_{\Omega}=\bigsqcup_{i,j}\gr_i^j \ar@<.5ex>[d]\ar@<-.5ex>[d]\ar[rr] &&PU(H)\ar@<.5ex>[d]\ar@<-.5ex>[d]\\
 \bigsqcup_{i}\Omega_{i}\ar[rr]&&\{e\}
}
\end{equation}
for an open cover  $\Omega= \{\Omega_i\}$ of $\go$.

 \begin{remark}
 The definition of generalized morphisms given in the last subsection was for two Lie groupoids. The group $PU(H)$ it is not precisely a Lie group but it makes perfectly sense to speak of generalized morphisms from Lie groupoids to this infinite dimensional "Lie" groupoid following exactly the same definition. Indeed, it is obviuos once one looks at equivalent datas (\ref{gpalpha}) and (\ref{galpha}) above (recall $PU(H)$ is considered with the topology induced by the norm topology on $U(H)$).
 \end{remark}
 
  \begin{remark} A twisting on a Lie groupoid $\gr \rightrightarrows M$ gives rise to an
  $U(1)$-central extension over the Morita equivalent groupoid $\gr_{\Omega}$ by
  pull-back the $U(1)$-central extension of $PU(H)$
  \[
  1\to U(1) \to U(H)  \to PU(H) \to 1.
  \]
  We will not call an $U(1)$-central extension of a Morita equivalent groupoid of 
  $ \gr$ a twisting on $\gr$ as in \cite{TXL}.  This is due to the fact that the associated principal  $PU(H)$-bundle might depend on the choice of Morita equivalence bibundles,  even though
  the isomorphism class of principal  $PU(H)$-bundle does not  depend on the choice of Morita equivalence bibundles.  It is important in applications to remember the  $PU(H)$-bundle 
  arising from a twisting, not just its isomorphism class. 
  \end{remark}

\begin{example} \label{example} We  give a list of  various twistings on  standard groupoids appearing  in this  paper.
 \begin{enumerate}
\item (Twisting on Lie groups) 
Let $G$ be a Lie group. Then $$G\rightrightarrows \{ e \}$$ is a Lie groupoid. A twisting on 
$G$ is given by a projective unitary representation 
\[
G\longrightarrow PU(H).
\]
\item (Twisting on manifolds)  Let $X$ be a $\ci$-manifold. We can consider the groupoid 
$$X\rightrightarrows X$$ where every morphism is the identity over $X$.  A twisting on $X$ is
given by a locally trivial principal $PU(H)$-bundle over $X$, or equivalently,  
 a twisting on $X$ is defined by a strict homomorphism 
\[
\xymatrix{
X_{\Omega}=\bigsqcup_{i,j} \Omega_{ij}  \ar@<.5ex>[d]\ar@<-.5ex>[d]\ar[rr] &&PU(H)\ar@<.5ex>[d]\ar@<-.5ex>[d]\\
 \bigsqcup_{i}\Omega_{i}\ar[rr]&&\{e\}.
}
\]
with respect to an open cover $\{\Omega_i\}$ of $X$, where $\Omega_{ij} =\Omega_i \cap\Omega_j$. 
Therefore, the restriction of a twisting $\sigma$ on a  Lie groupoid $\gr \rightrightarrows \go$
to its unit $\go$ defines a twisting  $\a_0$ on the manifold $\go$. 

\item (Orientation twisting) Let $X$ be a smooth  oriented manifold of dimension $n$. The tangent bundle $TX\to X$ defines
a natural generalized morphism 
\[
\xymatrix{
X\ar@{-->}[r] & SO(n).}
\]
Note that the fundamental representation of the Spin$^c$ group $Spin^c(n)$ gives rise to a commutative
diagram of Lie group homomorphisms
\[
\xymatrix{
Spin^c(n) \ar[d]   \ar[r] & U(\CC^{2^n}) \ar[d] \\
SO(n) \ar[r] & PU(\CC^{2^n}).}
\]
With a choice of inclusion $\CC^{2^n}$ into a Hilbert space $H$, we have a canoncial
twisting, called the orientation twisting  of $X$ associated to its tangent bundle, denoted by
\[\xymatrix{
o_{TX}:  X\ar@{-->}[r] & PU(H).}
\]

\item (Pull-back twisting) Given a twisting $\a$ on $ \gr$ and  for any generalized 
homomorphism $\phi: \hr \to \gr$, there is a pull-back twisting 
\[\xymatrix{
\phi^*\a:  \hr  \ar@{-->}[r]  & PU(H)}
\]
defined by the composition of $\phi$ and $\sigma$.  In particular, 
for a continuous map $\phi: X\to Y$, a twisting $\sigma$ on $Y$ gives a pull-back twisting 
$\phi^*\a$ on $X$. The principal $PU(H)$-bundle over $X$ defines by $\phi^*\a$ is
the pull-back of the  principal $PU(H)$-bundle on $Y$ associated to $\sigma$.

\item (Twisting on $G$-spaces)  Let $G$ be a Lie group acting by diffeomorphisms in a manifold $M$. The transformation groupoid associated to this action is
$$M \rtimes G\rightrightarrows M.$$ As a set $M\rtimes
G=M\times G $, and the maps are given by 
$s(x,g)=x\cdot g$, $r(x,g)=x$, the product given by 
$(x,g)\circ (x\cdot g,h)=(x,gh)$, the unit is $u(x)=(x,e)$ and with inverse $(x,g)^{-1}=(x\cdot g,g^{-1})$. The projection from $M \rtimes G $ to $G$ defines a strict homomorphism between groupoids
$M \rtimes G$ and $G$, then any projective representation $G\to PU(H)$ defines 
a twisting on $M \rtimes G$, which is the pull-back twisting from a twisting on  $G$.

\item (Twisting on pair groupoid)  Let $M$ be a $\ci$-manifold. We can consider the groupoid 
$$M\times M\rightrightarrows M$$ with $s(x,y)=y$, $r(x,y)=x$ and the product given by 
$(x,y)\circ (y,z)=(x,z)$.   Any locally trivial principal
$PU(H)$-bundle $P$  is trivial as the left action of  $M\times M\rightrightarrows M$ on $P$ gives a canonical trivialization 
\[
P \cong M \times PU(H).
\]
Hence,  any twisting on the pair 
groupoid $M \times M\rightrightarrows M$ is always trivial. 

\item (Twisting on fiber product groupoid)  Let $N\stackrel{p}{\rightarrow} M$ be a submersion. We consider the fiber product $N\times_M N:=\{ (n,n')\in N\times N :p(n)=p(n') \}$,which is a manifold because $p$ is a submersion. We can then take the groupoid 
$$N\times_M N\rightrightarrows N$$ which is  a subgroupoid of the pair groupoid 
$N\times N \rightrightarrows N$.  Note that this groupoid is in fact Morita equivalent to the groupoid $ M \rightrightarrows M$.  A twisting on  
 $N\times_M N\rightrightarrows N$ is  given by
a pull-back twisting from a   twisting on  $M$.

\item (Twisting on vector bundles) Let $\pi: E \to   X$ be a  vector bundle over a manifold $X$. We consider the groupoid 
$$E\rightrightarrows X$$ with  the source map $s(\xi)=\pi(\xi)$,  the target map $r(\xi)=\pi(\xi)$, the product   given by $\xi \circ \eta =\xi +\eta$. The unit is   zero section and the inverse is the additive inverse at each fiber. With respect to a trivialization cover  $\{E|_{\Omega_i} \cong \Omega_i \times V\}$ of $E$,  a  twisting on the groupoid 
$E\rightrightarrows X$ is given by  a $PU(H)$-valued cocycle
\[
g_{ij}:  \Omega_{ij}\times V \longrightarrow PU(H)
\]
satisfying  $g_{ij}(x, \xi) \cdot g_{jk}(x, \eta )  = g_{ik}(x, \xi+\eta)$. 
A  twisting on the groupoid 
$E\rightrightarrows X$ is a pull-back twisting $\pi^*\a$  from a twisting $\sigma$ on $X$ if 
 $g_{ij}:  \Omega_{ij}\times V \to PU(H)$ is constant in $V$.    Note that the pull-back twisting  agrees  with the 
  pull-back twisting  on $E$ by $\pi$ as a topological space.

\end{enumerate}
\end{example}
 
 In this paper, we will mainly deal  with the holonomy groupoids associated to regular foliations.  Let $M$ be a manifold of dimension $n$. Let $F$ be a subvector bundle of the tangent bundle $TM$.
We say that $F$ is integrable if  
$\ci(F):=\{ X\in \ci(M,TM): \forall x\in M, X_x\in F_x\}$ is a Lie subalgebra of $\ci(M,TM)$. This induces a partition  of $M$ in embedded submanifolds (the leaves of the foliation), given by the solution of integrating $F$. 

The holonomy groupoid of $(M,F)$ is a Lie groupoid  
$$\gr_{M } \rightrightarrows M$$ with Lie algebroid $A\gr=F$ and minimal in the following sense: 
any Lie groupoid integrating the foliation, that is having $F$ as Lie algebroid,  contains an open subgroupoid which maps onto the holonomy groupoid by a smooth morphism of Lie groupoids. 
The holonomy groupoid was constructed by Ehresmann \cite{Ehr} and Winkelnkemper \cite{Win} (see also  \cite{Candel}, \cite{God}, \cite{Pat}).
  
\begin{definition}\label{twistingleaves}
(Twisting on the space of leaves of a foliation)
Let $(M,F)$ be a regular foliation with holonomy groupoid $\gr_{M }$. A twisting on the space of leaves  is by definition a twisting on the holonomy groupoid $\gr_{M }$. We will often use the notation 
\[\xymatrix{
M/F \ar@{-->}[r] & PU(H)}
\]
for the correspondant generalized morphism.
\end{definition} 
Notice that by definition a twisiting on the spaces of leaves is a twisting on the base $M$ which admits a compatible action of the holonomy groupoid. It is however not enough to have a twisting on base which is leafwise constant, see for instance remark 1.4 (c) in \cite{HS}.

    \subsection{Twistings on  tangent groupoids}

In this subsection, we review the notion of Connes' tangent groupoids from deformation to the normal cone point of view, and discuss the induced twistings on tangent groupoids.

\subsubsection{Deformation to the normal cone}

The tangent groupoid is a particular case of a geometric construction that we describe here.

Let $M$ be a $\ci$ manifold and $X\subset M$ be a $\ci$ submanifold. We denote
by $\Nb_{X}^{M}$ the normal bundle to $X$ in $M$.
We define the following set
\begin{align}
\Dnc_{X}^{M}:= \Nb_{X}^{M} \times {0} \bigsqcup M \times \Rr^* 
\end{align} 
The purpose of this section is to recall how to define a $\ci$-structure in $\Dnc_{X}^{M}$. This is more or less classical, for example
it was extensively used in \cite{HS}.

Let us first consider the case where $M=\Rr^p\times \Rr^q$ 
and $X=\Rr^p \times \{ 0\}$ ( here we
identify  $X$ canonically with $ \Rr^p$). We denote by
$q=n-p$ and by $\Dnc_{p}^{n}$ for $\Dnc_{\Rr^p}^{\Rr^n}$ as above. In this case
we   have that $\Dnc_{p}^{n}=\Rr^p \times \Rr^q \times \Rr$ (as a
set). Consider the 
bijection  $\psi: \Rr^p \times \Rr^q \times \Rr \rightarrow
\Dnc_{p}^{n}$ given by 
\begin{equation}\label{psi}
\psi(x,\xi ,t) = \left\{ 
\begin{array}{cc}
(x,\xi ,0) &\mbox{ if } t=0 \\
(x,t\xi ,t) &\mbox{ if } t\neq0
\end{array}\right.
\end{equation}
whose  inverse is given explicitly by 
$$
\psi^{-1}(x,\xi ,t) = \left\{ 
\begin{array}{cc}
(x,\xi ,0) &\mbox{ if } t=0 \\
(x,\frac{1}{t}\xi ,t) &\mbox{ if } t\neq0
\end{array}\right.
$$
We can consider the $\ci$-structure on $\Dnc_{p}^{n}$
induced by this bijection.

We pass now to the general case. A local chart 
$(\Uo,\phi)$ in $M$ is said to be a $X$-slice if 
\begin{itemize}
\item[1)]$\phi : \Uo  \rightarrow U \subset \Rr^p\times \Rr^q$ is a diffeomorphsim. 
\item[2)]  If  $V =U \cap (\Rr^p \times \{ 0\})$, then
$\phi^{-1}(V) =   \Uo \cap X$ , denoted by $\Vo$.
\end{itemize}
With this notation, $\Dnc_{V}^{U}\subset \Dnc_{p}^{n}$ as an
open subset. We may define a function 
\begin{equation}\label{phi}
\tilde{\phi}:\Dnc_{\Vo}^{\Uo} \rightarrow \Dnc_{V}^{U} 
\end{equation}
in the following way: For $x\in \Vo$ we have $\phi (x)\in \Rr^p
\times \{0\}$. If we write 
$\phi(x)=(\phi_1(x),0)$, then 
$$ \phi_1 :\Vo \rightarrow V \subset \Rr^p$$ 
is a diffeomorphism. We set 
$\tilde{\phi}(v,\xi ,0)= (\phi_1 (v),d_N\phi_v (\xi ),0)$ and 
$\tilde{\phi}(u,t)= (\phi (u),t)$ 
for $t\neq 0$. Here 
$d_N\phi_v: N_v \rightarrow \Rr^q$ is the normal component of the
 derivative $d\phi_v$ for $v\in \Vo$. It is clear that $\tilde{\phi}$ is
 also a  bijection (in particular it induces a $C^{\infty}$ structure on $\Dnc_{\Vo}^{\Uo}$). 
Now, let us consider an atlas 
$ \{ (\Uo_{\sigma},\phi_{\sigma}) \}_{\sigma \in \Delta}$ of $M$
 consisting of $X-$slices. Then the collection $ \{ (\Dnc_{\Vo_{\sigma}}^{\Uo_{\sigma}},\tilde{\phi}_{\sigma})
  \} _{\sigma \in \Delta }$ is a $\ci$-atlas of
  $\Dnc_{X}^{M}$ (Proposition 3.1 in \cite{Ca4}).

\begin{definition}[Deformation to the normal cone]
Let $X\subset M$ be as above. The set
$\Dnc_{X}^{M}$ equipped with the  $C^{\infty}$ structure
induced by the atlas of  $X$-slices is called
 the deformation to the  normal cone associated  to   the embedding
$X\subset M$. 
\end{definition}


One important feature about the deformation to the normal cone is the functoriality. More explicitly,  let
 $f:(M,X)\rightarrow (M',X')$
be a   $\ci$ map   
$f:M\rightarrow M'$  with $f(X)\subset X'$. Define 
$ \Dnc(f): \Dnc_{X}^{M} \rightarrow \Dnc_{X'}^{M'} $ by the following formulas: \begin{enumerate}
\item[1)] $\Dnc(f) (m ,t)= (f(m),t)$ for $t\neq 0$, 

\item[2)]  $\Dnc(f) (x,\xi ,0)= (f(x),d_Nf_x (\xi),0)$,
where $d_Nf_x$ is by definition the map
\[  (\Nb_{X}^{M})_x 
\stackrel{d_Nf_x}{\longrightarrow}  (\Nb_{X'}^{M'})_{f(x)} \]
induced by $ T_xM 
\stackrel{df_x}{\longrightarrow}  T_{f(x)}M'$.
\end{enumerate}
 Then $\Dnc(f):\Dnc_{X}^{M} \rightarrow \Dnc_{X'}^{M'}$ is a $\ci$-map (Proposition 3.4 in \cite{Ca4}). In the language of categories, the deformation to the normal cone  construction defines a functor
\begin{equation}\label{fundnc}
\Dnc: \mathscr{C}_2^{\infty}\longrightarrow \mathscr{C}^{\infty} ,
\end{equation}
where $\mathscr{C}^{\infty}$ is the category of $\ci$-manifolds and $\mathscr{C}^{\infty}_2$ is the category of pairs of $\ci$-manifolds.

\begin{proposition}\label{twisting:normal} Given a twisting $\a: M \to PU(H)$ as  a principal $PU(H)$-bundle
over $M$,  there is a canonical twisting   on $ \Dnc_{X}^{M} $  which restricts to the pull-back twisting  $\pi^*(\a|_X)$ on $ \Nb_{X}^{M}$ by
the map $\pi: \Nb_{X}^{M} \to X$. 
\end{proposition}

\begin{proof} Under the identification of the normal bundle  $\Nb_{X}^{M}$ with a tubular neighborhood of $X$ in $M$, we have the following diagram of maps
\[
\xymatrix{ \Nb_{X}^{M} \ar[rd]^\iota \ar[d]^\pi & \\
X\ar[r]^i  \ar@/{}^{2pc}/[u]^{0_X} & M
}
\]
where $0_X$ is the zero section of the normal bundle. Note that  $\pi\circ  0_X =Id_X$, $i= \iota\circ 0_X$, and $0_X\circ \pi$ is homotopic to the identity map
on $\Nb_{X}^{M}$, which imply that  
\[
\pi^* \circ i^* \a =  \pi^* \circ 0_X^*  \circ \iota^*\a  = (0_X \circ \pi)^* \circ \iota^*\a
\]
is homotopic to $\iota^*\a$. Hence, the principal $PU(H)$-bundle over 
$\Nb_{X}^{M}$
associated to $\iota^*\a$ is isomorphic to 
the pull-back principal $PU(H)$-bundle on  $\Nb_{X}^{M}$  associated to $\pi^* (\a|_X)$.
\end{proof}

\subsubsection{The tangent groupoid of a groupoid}

\begin{definition}[Tangent groupoid]
Let $\gr \rightrightarrows \go $ be a Lie groupoid. $\it{The\, tangent\,
groupoid}$ associated to $\gr$ is the groupoid that has 
\[
\Dnc_{\go}^{\gr} = \Nb_{\go}^{\gr} \times \{0\}\bigsqcup \gr\times \Rr^*
\]
 as the set of arrows and  $\go \times \Rr$ as the units, with:
 \begin{enumerate}
\item  $s^T(x,\eta ,0) =(x,0)$ and $r^T(x,\eta ,0) =(x,0)$ at $t=0$.
\item   $s^T(\gamma,t) =(s(\gamma),t)$ and $r^T(\gamma,t)
  =(r(\gamma),t)$ at $t\neq0$.
\item  The product is given by
  $m^T((x,\eta,0),(x,\xi,0))=(x,\eta +\xi ,0)$ and  $m^T((\gamma,t), 
  (\beta ,t))= (m(\gamma,\beta) , t)$ if $t\neq 0 $ and 
if $r(\beta)=s(\gamma)$.
\item The unit map $u^T:\go \rightarrow \gr^T$ is given by
 $u^T(x,0)=(x,0)$ and $u^T(x,t)=(u(x),t)$ for $t\neq 0$.
\end{enumerate}
We denote $\gr^{T}= \Dnc_{\go}^{\gr}$ and $A\gr =\Nb_{\go}^{\gr}  $ as a vector bundle over $\gr^{(0)}$. Then we have a family of Lie groupoids parametrized by $\Rr$, which itself is a Lie groupoid
\[
\gr^T = A\gr \times \{0\}  \bigsqcup \gr\times \Rr^* \rightrightarrows \go\times \Rr.
\]
\end{definition} 
As a consequence of the functoriality of the deformation to the normal cone,
one can show that the tangent groupoid is in fact a Lie
groupoid compatible with the Lie groupoid structures of $\gr$ and $A\gr$ (considered as a Lie groupoid with its vector bundle structure). 
Indeed, it is immediate that if we identify in a
canonical way $\Dnc_{\go}^{\gr^{(2)}}$ with $(\gr^T)^{(2)}$, then 
$$ m^T=\Dnc(m),\, s^T=\Dnc(s), \,  r^T=\Dnc(r),\,  u^T=\Dnc(u)$$
where we are considering the following pair morphisms:
\[\begin{array}{l}
m:( \gr^{(2)},\go)\rightarrow (\gr,\go ), \nonumber
\\
s,r:(\gr ,\go) \rightarrow (\go,\go),\nonumber 
\\
u:(\go,\go)\rightarrow (\gr,\go ).\nonumber
\end{array}
\]

\begin{example}\label{exgt}
 \begin{enumerate}
\item The tangent groupoid of a group: Let $G$ be a Lie group. We consider it as a groupoid with one element, $G\rightrightarrows \{ e\}$, as above. By definition $AG=T_eG$, {\it i.e.},
  The Lie algebroid coincides with the Lie algebra 
$\mathfrak{g}$
 of the group $G$. Hence, the tangent groupoid is
$$G^T:=\mathfrak{g}\times \{ 0\}\bigsqcup G\times \Rr^* \rightrightarrows G\times \Rr$$ 
  
  \item The tangent groupoid of a $\ci$-manifold: Let $M$ be a $\ci$-manifold. 
  We consider the associated pair groupoid
$M\times M \rightrightarrows M$. In this case, the Lie algebroid can be  identified with $TM$ 
and the tangent groupoid take the following form
$$\gr_M^T=TM \times \{ 0\} \bigsqcup M\times M \times (0,1] \rightrightarrows M \times \Rr .$$

\item  The Thom groupoid: Let $N\stackrel{p}{\rightarrow} T$ be a vector bundle over a smooth manifold 
$T$. Consider the fiberwise  pair groupoid over $T$,
\begin{equation}\label{NXTN}
N\times_T N \rightrightarrows N.
\end{equation}
The Lie algebroid is $N\oplus N$ as a vector bundle over $N$, Then the tangent groupoid, denoted by
$\gr_{\bf{Thom}}$, takes the following form
$$\gr_{\bf{Thom}}=N\oplus N \times \{ 0\} \bigsqcup
N\times_T N \times \Rr^* \rightrightarrows N \times \Rr.$$
See \cite{DLN} for a motivation for the name of this groupoid. 

\item Tangent groupoid of a holonomy groupoid:  Let $\gr_{M/F}$ be the holonomy groupoid
of a foliated manifold $M/F$. Then the tangent groupoid is given by
\[
\gr^T_{M/F} = F\times\{0\} \bigsqcup \gr_{M/F} \rightrightarrows M \times \Rr.
\]
\end{enumerate}

\end{example}


The following proposition could be used alternatively to construct the twisted index morphism and to develop the twisted index theory. However in the sequel we will rather use the central extension approach because the proofs are simpler, and because we obtain naturally the twisted index theory as a factor of an $S^1$-equivariant index theory. It gives though a nice geometric idea of the strict deformation quantization methods.

\begin{proposition}\label{twist:tangentgoid}
Given a twisting $\sigma$ on  a Lie groupoid $\gr$, there is a canonical 
twisting $\a^T$ on its tangent groupoid $\gr^T$ such that 
\[
\a^T|_{A\gr} =  \pi^* (\a|_{\go})
\]
where $\pi: A\gr \to \go$ is the projection.
\end{proposition}

\begin{proof}
The twisting $\sigma$ on $\gr$ is given by a $PU(H)$-principal bundle $P_{\sigma}$ with a compatible left action of $\gr$. By definition of the groupoid action, the units of the groupoid act as identities, hence we can consider the action as a $\ci$-morphism in the category of $\ci$-pairs: 
$$(\gr\times_MP_{\sigma},M\times_MP_{\sigma})\longrightarrow (P_{\sigma},P_{\sigma}).$$
We can then apply the deformation to the normal cone functor to obtain the desire $PU(H)$-principal bundle with a compatible $\gr^T$-action, which gives the desired twisting.
\end{proof}

\subsubsection{The tangent groupoid of a groupoid immersion}\label{tgrpdimm}

We  briefly discuss here the tangent groupoid of an immersion of groupoids which is called  the normal groupoid in   \cite{HS}.

Consider an immersion of Lie groupoids $\gr_1\stackrel{\varphi}{\hookrightarrow}\gr_2$. Let 
$\gr^N_1=\Nb_{\gr_1}^{\gr_2}$ be the total space of the normal bundle to $\varphi$, and $(\gr_{1}^{(0)})^N$ be the total space of the normal bundle to $\varphi_0: \go_1 \to \go_2$. The deformation to the normal bundle construction allows us to consider a $C^{\infty}$ structure on 
$$\gr_{\varphi}:=\gr^N_1\times \{0\}\bigsqcup \gr_2\times \mathbb{R}^*,$$
such that $\gr^N_1\times \{0\}$ is a closed saturated submanifold and so $\gr_2\times \mathbb{R}^*$ is an open submanifold.

\begin{remark}\label{normalgrpd}
Let us observe that $\gr^N_1\rightrightarrows (\gr_{1}^{(0)})^N$ is a Lie groupoid with the following structure maps:
\begin{enumerate}
\item The source map is the derivation in the normal direction 
$d_Ns:\gr^N_1\rightarrow (\gr_{1}^{(0)})^N$ of the source map (seen as a pair of maps) $s:(\gr_2,\gr_1)\rightarrow (\gr_{2}^{(0)},\gr_{1}^{(0)})$ and similarly for the target map.
\item The product map is the derivation in the normal direction $d_Nm:(\gr^N_1)^{(2)}
\rightarrow \gr_1^N$ of the product map   $m:(\gr_{2}^{(2)},\gr_{1}^{(2)})\rightarrow (\gr_2,\gr_1)$
\end{enumerate}
\end{remark}

The following results are   an immediate consequence of the functoriality of the deformation to the normal cone construction.

\begin{proposition}[Hilsum-Skandalis, \cite{HS}]\label{HSimmer}
Let $\gr_{\varphi_0}:=(\gr_{1}^{(0)})^N\times \{0\}\bigsqcup \gr_{2}^{(0)}\times \mathbb{R}^*$ be the deformation to the normal cone of  the  pair $(\gr_{2}^{(0)},\gr_{1}^{(0)})$. The groupoid
\begin{equation}
\gr_{\varphi}\rightrightarrows\gr_{\varphi_0}
\end{equation}
with structure maps compatible  with the ones of the groupoids $\gr_2\rightrightarrows \gr_{2}^{(0)}$ and $\gr_1^N\rightrightarrows (\gr_{1}^{(0)})^N$, is a Lie groupoid with $C^{\infty}$-structures coming from  the deformation to the normal cone.
\end{proposition}

\begin{proposition} Given an immersion of Lie groupoids $\gr_1\stackrel{\varphi}{\hookrightarrow}\gr_2$  and a twisting $\sigma$ on $\gr_2$.  There is a canonical twisting $\a_\varphi$ on the Lie groupoid 
$\gr_{\varphi} \rightrightarrows \gr_{\varphi_0}$,  extending the pull-back  twisting on $\gr_2\times \Rr^*$ from $\sigma$.
\end{proposition}

\section{Analytic index morphism for a twisted Lie groupoid}

In this section, we first review the analytic index morphism for any Lie groupoid, and then develop
an analytic index morphism for a Lie groupoid with a twisting. 
\subsection{The case of Lie groupoids}

\subsubsection{The convolution  $C^*$-algebra of a Lie groupoid}

We recall  how to define the reduced and maximal  $C^*$-algebras of a Lie groupoid 
$\gr\rightrightarrows M$, $C_r^*(\gr)$ and $C^*(\gr)$. 
 
We start with the reduced $C^*$-algebra. The basic elements of $C_r^*(\gr)$ are smooth half densities with compact support on  $\gr$ in $ \ci_c(\gr,\Omega^{\frac{1}{2}})$ where $ \Omega^{\frac{1}{2}}$ is the real  line bundle over $\gr$ given by 
$ s^*(\Omega_0^{\frac{1}{2}})\otimes t^*(\Omega_0^{\frac{1}{2}})$ and $\Omega_)^{\frac{1}{2}}$ denote the half density bundle associated to $A\gr$.   In particular, $\Omega^{\frac{1}{2}}_{\gamma}=(\Omega_0)^{\frac{1}{2}}_{s(\gamma)}\otimes (\Omega_0)^{\frac{1}{2}}_{t(\gamma)}$. 
The convolution product on $\ci_c(\gr,\Omega^{\frac{1}{2}})$ is given by
$$(f*g)(\gamma)
=\int_{\gamma_1\cdot\gamma_2=\gamma}
f(\gamma_1)g(\gamma_2),$$
where the integration is performed over the resulting $1$-density.
The $*$-operation is simply given by $f^*(\gamma)=\overline{f(\gamma^{-1})}$.
 
For every $x\in M$, there is a representation $\pi_x$ of $\ci_c(\gr,\Omega^{\frac{1}{2}})$ on $L^2(\gr_x,\Omega^{\frac{1}{2}})$ given by
$$\pi_x(f)(\xi)(\gamma)=\int_{\gamma_1\cdot\gamma_2=\gamma}
f(\gamma_1)\xi(\gamma_2)$$
 
The norm $\| f\|:=sup_{x\in M}{\pi_x(f)}$ is well defined (see \cite{Concg} proposition 5.3 chapter 2 or \cite{Ren} for more details) and the reduced $C^*$algebra
$C^*_r(\gr)$ is the $C^*$-completion with respect to this norm. For obtain the maximal, one has to take the completion with respect to the norm $\| f\|:=sup_{\pi}{\pi(f)}$  over all  involutive Hilbert  representations of $\ci_c(\gr,\Omega^{\frac{1}{2}})$.

The line bundle $\Omega^{\frac{1}{2}}$ is trivial, even if there is not a canonical trivialization. We have therefore a non canonical isomorphism between the sections $\ci_c(\gr,\Omega^{\frac{1}{2}})$ and the compactly supported functions $\ci_c(\gr)$. In the sequel we will often use the notation $\ci_c(\gr)$ even if we are thinking on half densities. 

\subsubsection{Analytic index morphism for Lie groupoids}

A $\gr$-$\it{pseudodifferential}$ $\it{operator}$ is a family of
pseudodifferential operators $\{ P_x\}_{x\in \go} $ acting in
$\ci_c(\gr_x)$ such that if $\gamma \in \gr $ and
$$U_{\gamma}:\ci_c(\gr_{s(\gamma)}) \rightarrow \ci_c(\gr_{r(\gamma)}) $$
the induced operator, then we have the following compatibility condition 
$$ P_{r(\gamma)} \circ U_{\gamma}= U_{\gamma} \circ P_{s(\gamma)}.$$
We also admit, as usual, operators acting in sections of a complex  vector bundle  $E\rightarrow \go$.
There is also a differentiability condition with respect to $x$ that can
be found in \cite{NWX}.

In this work we are going to work exclusively with uniformly supported operators, let us recall this notion. Let $P=\{ P_x,x\in \go\}$ be a $\gr$-operator, we denote by $k_x$ the Schwartz kernel of $P_x$. Let 
$$supp\, P:=\overline{\cup_xsupp \,k_x}, \text{ and }$$
$$supp_{\mu}P:=\mu_1(supp\,P),$$ where $\mu_1(g',g)=g'g^{-1}$. 
We say that $P$ is uniformly supported if $supp_{\mu}P$ is compact.

We denote by $\Psi^m(\gr,E)$ the space of uniformly supported $\gr$-operators  of order $m$, acting on sections of a complex vector bundle $E$. We denote also
\begin{center}
$\Psi^{\infty}(\gr,E)=\bigcup_m \Psi^m(\gr,E)$ and  $\Psi^{-\infty}(\gr,E)=\bigcap_m \Psi^m(\gr,E).$
\end{center}

The composition of two such operators is again of this kind (Lemma 3, \cite{NWX}). 
In fact, $\Psi^{\infty}(\gr,E)$ is a filtered algebra (Theorem 1, rf.cit.), $\it{i.e.}$, 
\begin{center}
$\Psi^{m}(\gr,E)\Psi^{m'}(\gr,E)\subset \Psi^{m+m'}(\gr,E).$
\end{center}
In particular,  $\Psi^{-\infty}(\gr,E)$ is a bilateral ideal.

\begin{remark}
The choice on the support can be justified  by  the fact that  
$\Psi^{-\infty}(\gr,E)$ is identified with 
$\ci_c(\gr,End(E))$, thanks the Schwartz kernel theorem. 
\end{remark}

The notion of principal symbol extends also to this setting. Let us denote by 
$\pi: A^*\gr \rightarrow \go$ the projection. For $P=\{ P_x,x\in \go\} \in \Psi^{m}(\gr,E,E)$, 
the principal symbol of $P_x$, $\sigma_m(P_x)$, is a $\ci$ section of the vector bundle 
$Hom(\pi_x^*r^*E, \pi_x^*r^*E)$ over $T^*\gr_x$ (where 
$\pi_x:T^*\gr_x \rightarrow \gr_x$), such that at each fiber the morphism is homogeneous of degree $m$ 
(see \cite{AS} for more details). 
There is a section  
$\sigma_m(P)$ of $Hom(\pi^*E, \pi^*E)$ over $A^*\gr$ 
such that
\begin{equation}\label{gpsym}
\sigma_m(P)(\xi)=\sigma_m(P_x)(\xi)\in Hom(E_x,E_x) \text{ if  } \xi \in A^*_x\gr
\end{equation}
This   induces a unique surjective linear map
\begin{equation}\label{gpsymap}
\sigma_m:\Psi^{m}(\gr,E)\rightarrow \sw^m(A^*\gr,Hom(E,E)),
\end{equation}
with kernel $\Psi^{m-1}(\gr,E)$ (see for instance 
Proposition 2 \cite{NWX}) and where $\sw^m(A^*\gr,Hom(E,E))$ denotes the sections of the vector  bundle
$Hom(\pi^*E,\pi^*E)$ over $A^*\gr$, 
homogeneous of degree $m$ along  each fiber.

\begin{definition} 
Let $P=\{P_x,x\in \go\}$ be a $\gr$-pseudodifferential operator. 
 $P$ is elliptic if   $P_x$ is elliptic for each $x$. 
We denote by $Ell(\gr)$ the set of  elliptic $\gr$-pseudodifferential  operators.
\end{definition}

The linear map  (\ref{gpsymap})  for  elliptic   $\gr$-pseudodifferential operators defines a principal symbol map:
\begin{equation}\label{ellgsymb}
Ell(\gr)\stackrel{[\sigma]}{\longrightarrow}K^0(A^*\gr).
\end{equation}

Connes in  \cite{Coinc} proved that if $P=\{P_x,x\in \go\} \in Ell(\gr)$, then it exists 
$Q\in  \Psi^{-m}(\gr,E)$ such that
\begin{center}
$Id_E-PQ\in \Psi^{-\infty}(\gr,E)$  and  $Id_E-QP\in \Psi^{-\infty}(\gr,E)$,
\end{center}
where $Id_E$ denotes the identity operator over $E$. In other words, $P$ defines 
 an element in $K_0(\cg)$, when $E$ is trivial,  given by  
\begin{equation}\label{index}
\left[T\left( 
\begin{array}{cc}
1 & 0\\
0 & 0
\end{array}
\right)
T^{-1}\right]- \left[ \left( 
\begin{array}{cc}
1 & 0\\
0 & 0
\end{array}
\right) \right] \in K_0(\widetilde{\cg}),
\end{equation}
where $1$ is the unit in $\widetilde{\cg}$ (the unitarisation of $\cg$), and where $T$ is given by 
$$
T=\left( 
\begin{array}{cc}
(1-PQ)P+P & PQ-1\\
1-QP & Q
\end{array}
\right)
$$
with inverse
$$T^{-1}=\left( 
\begin{array}{cc}
Q & 1-QP\\
PQ-1 & (1-PQ)P+P
\end{array}
\right).$$

If $E$ is not trivial we obtain in the same way an element of 
$K_0(\ci_c(\gr,Hom(E,F)))\cong  K_0(\cg)$ since $\ci_c(\gr,Hom(E,F)))$ is Morita
equivalent to $\cg$.

\begin{definition}($\gr$-Index)
Let $P$ be an elliptic  $\gr$-pseudodifferential   operator. We denote by 
\[
\ind_\gr (P)  \in K_0(\cg)
\]
the element defined by $P$ as above, called the $\gr$-index of $P$.  The $\gr$-index  defines a correspondence
\begin{equation}\label{ellgind}
\xymatrix{ Ell(\gr)\ar[rr]^{\ind_\gr}  && K_0(\cg)}.
\end{equation}
\end{definition}

 

Consider the morphism 
\begin{equation}
K_0(\cg)\stackrel{j}{\longrightarrow} K_0(C^*_r(\gr))
\end{equation}
induced by the inclusion $\cg \subset C^*_r(\gr)$, then the composition 
$$Ell(\gr)\stackrel{\ind_\gr}{\longrightarrow}K_0(\cg)\stackrel{j}{\longrightarrow} K_0(C^*_r(\gr))$$ factors through the principal symbol class. In other words, we have the following commutative diagram
\[
\xymatrix{
Ell(\gr) \ar[r]^{\ind_\gr}\ar[d]_{[\sigma]}& K_0(\cg)\ar[d]^j \\
K^0(A^*\gr) \ar[r]_{\ind_a} & K_0(C^*_r(\gr)).
}
\]
  The  resulting morphism
\begin{equation}\xymatrix{
K^0(A^*\gr) \ar[rr]^{\ind_{a, \gr}}&& K_0(C^*_r(\gr))}
\end{equation}
is called the analytic index morphism of $\gr$. In fact,  $\ind_{a, \gr}$ is the index morphism associated to the exact sequence of $C^*$-algebras (\cite{Coinc}, \cite{CS}, \cite{MP}, \cite{NWX})  
\begin{equation}\label{gpdose}
0\rightarrow C^*_r(\gr)\longrightarrow \overline{\Psi^0(\gr)}\stackrel{\sigma}{\longrightarrow}C_0(S(A^*\gr))\rightarrow 0
\end{equation}
where $\overline{\Psi^0(\gr)}$ is a certain $C^*$-completion of $\Psi^0(\gr)$, $S(A^*\gr)$ is the sphere  bundle of $A^*\gr$ and $\sigma$ is the extension of the principal symbol.


We remark that  this index morphism $\ind_{a, \gr}$ can also be constructed using the tangent groupoid and its $C^*$-algebra. We briefly recall this construction from \cite{MP}. 

Notice that  the evaluation morphisms extend to the $C^*$-algebras as in  \cite{Ren}:
 $$C^*(\gr^T) \stackrel{ev_0}{\longrightarrow}
C^*(A\gr) \text{ for $t=0$, and }$$ $$C^*(\gr^T) \stackrel{ev_t}{\longrightarrow}
C^*(\gr) \text{ for } t\neq 0.$$
Moreover, since $\gr \times (0,1]$ is an open saturated subset of $\gt$ and $A\gr$ a saturated closed subset, we have the following exact sequence (\cite{Ren,HS})
\begin{equation}\label{segt}
0 \rightarrow C^*(\gr \times (0,1]) \longrightarrow C^*(\gr^T)
\stackrel{ev_0}{\longrightarrow}
C^*(A\gr) \rightarrow 0.
\end{equation}
Now, the $C^*$-algebra $C^*(\gr \times (0,1])\cong C_0((0,1],C^*(\gr))$ is contractible. 
This implies that the groups $K_i(C^*(\gr \times (0,1]))$ vanish, for $i=0,1$. 
Then, applying the $K-$theory  functor to the exact sequence above, we obtain that
$$K_i(C^*(\gr^T)) \stackrel{(ev_0)_*}{\longrightarrow}
K_i(C^*(A\gr))$$ is an isomorphism, for $i=0,1$. In \cite{MP},
Monthubert-Pierrot show that
\begin{equation}\label{defind}
\ind_{a, \gr}=(ev_1)_*\circ
(ev_0)_*^{-1},
\end{equation}
modulo the Fourier isomorphism   $C^*(A\gr)\cong
C_0(A^*\gr)$, when we consider the $K_0$ evaluations (see also \cite{HS} and \cite{NWX}). As usual we consider the index with values in the reduced $C^*$-algebra by taking the canonical morphism 
$K_{0}(C^*(\gr))\longrightarrow K_{0}(C_r^*(\gr))$.
Putting this in a commutative diagram, we have
\begin{equation}\label{eldiagramaevaluadoestrella}
\xymatrix{
&K_0(C^*(\gr^T))\ar[ld]_{(ev_0)_*}^{\cong} \ar[rd]^{(ev_1)_*}&\\
K^0(A^*\gr)\ar[rr]_{\ind_{a, \gr}}&&K_{0}(C_r^*(\gr)). 
}
\end{equation}
In summary, the algebra $C^*(\gr^T)$ is a strict deformation quantization of $C_0(A^*\gr)$, and the analytic index morphism of $\gr$ can be constructed by means of this deformation.

\subsubsection{Analytic indices morphisms for Lie groupoid immersions} \ 

 For Lie groupoid immersions,   Hilsum-Skandalis in \cite{HS} shown that   the same methods as above  can be applied to define an index morphism (or even a KK-element) associated to an  immersion of Lie groupoids $\varphi:\gr_1\hookrightarrow \gr_2$, where we assume as in section III  of  \cite{HS} that $\gr_1$ is amenable (so $C^*(\gr_1)$  is nuclear). Indeed, if we consider
$$\gr_{\varphi}:=\gr^N_1\times \{0\}\bigsqcup \gr_2\times (0,1]$$
and the short exact sequence
\begin{equation}\label{segtimm}
0 \rightarrow C^*(\gr_2 \times (0,1]) \longrightarrow C^*(\gr_{\varphi})
\stackrel{ev_0}{\longrightarrow}
C^*(\gr_1^N) \rightarrow 0,
\end{equation}
we can define the index morphism
\begin{equation}\label{HSindex}
\ind_{\varphi}:K_*(C_r^*(\gr_1^N))\longrightarrow K_*(C_r^*(\gr_2)),
\end{equation}
as the induced deformation morphism $\ind_{\varphi}:=p_*\circ(ev_1)_*\circ(ev_0)_{*}^{-1}$, where $p:C^*(\gr_2)\to C_r^*(\gr_2)$ is the canonical morphism. Here we use  the fact that $\gr_1^N$ is also amenable and so $C_r^*(\gr_1^N)=C^*(\gr_1^N)$.

\subsection{The case of twisted groupoids}

\subsubsection{Twisted K-theory of a  Lie groupoid with a twisting}

Let $(\gr,\sigma)$ be a twisted groupoid. With respect to a covering  $\Omega = \{\Omega_i\}$ of $\go$, the twisting $\sigma$ is given by a strict morphism of groupoids 
$$ \sigma: \gr_{\Omega}   \longrightarrow PU(H),$$
where $\gr_{\Omega}$ is the covering groupoid associated to $\Omega$. 
Consider the central extension of groups
$$S^1 \longrightarrow U(H) \longrightarrow PU(H),$$
 we can pull it back to get a $S^1$-central extension of Lie groupoid $R_{\sigma}$  over $\gr_{\Omega}$ 
\begin{equation}
\xymatrix{
  S^1\ar[d]\ar[r]& S^1\ar[d]\\
R_{\sigma}\ar[d]\ar[r]&U(H)\ar[d]\\
\gr_{\Omega}\ar[r]_-{\sigma}&PU(H)\\
}
\end{equation}
In particular, $R_{\sigma}\rightrightarrows \bigsqcup_i\Omega_i$ is a Lie groupoid and $R_{\sigma}\longrightarrow \gr_{\Omega}$ is a $S^1$-principal bundle.

We recall the definition of the convolution  algebra and the $C^*$-algebra of a twisted Lie groupoid $(\gr, \a)$ \cite{Ren87,TXL}:

\begin{definition}
Let $R_{\sigma}$ be the $S^1$-central extension of groupoids associated to a twisting $\sigma$. The convolution algebra of $(\gr, \a)$ is by definition the following sub-algebra of 
$C_{c}^{\infty}(R_{\sigma})$: 
\begin{equation}
C_{c}^{\infty}(\gr, \a)=\{f\in C_{c}^{\infty}(R_{\sigma}): f(\tilde{\gamma} \cdot \lambda)=\lambda^{-1}\cdot f(\tilde{\gamma}), \forall \tilde{\gamma}\in R_{\sigma },\forall \lambda \in S^1\}.
\end{equation}
The reduced (maximal resp.) $C^*$-algebra of $(\gr, \sigma)$, denoted by $C^*_r(\gr,  \sigma)$ ($C^*(\gr, \sigma)$ resp.),  is the completion of 
$C_{c}^{\infty}(\gr, \a)$ in $C_r^*(R_{\sigma})$ ($C^*(R_{\sigma})$ resp.).
\end{definition}

Let $L_{\sigma}:=R_{\sigma}\times_{S^1}\mathbb{C}$ be the  complex line bundle  over $\gr_\Omega$ which can be considered as a Fell bundle (using the groupoid structure of $R_{\sigma}$) over $\gr_{\Omega}$. 
In fact, the algebra of sections of this Fell bundle $C_{c}^{\infty}(\gr_{\Omega },L_{\sigma})$ is isomorphic to $C_{c}^{\infty}(\gr, \a)$, and the same is true for the $C^*$-algebras, $C^*(\gr_{\Omega}, L_{\sigma})\cong C^*(\gr, \a)$ (see (23) in \cite{TXL} for an explicit isomorphism).

\begin{remark}\label{decomposition}\cite{TXL}. 
Given the extension $R_{\sigma}$ as above, the $S^1$-action on $R_{\sigma}$ induces a $\mathbb{Z}$-gradation in $C_r^*(R_{\sigma})$ (Proposition 3.2, ref.cit.). More precisely,  we have
\begin{equation}\label{gradext}
C_r^*(R_{\sigma})\cong \bigoplus_{n\in \mathbb{Z}} C_r^*(\gr, \sigma^n)
\end{equation}
where $C_r^*(\gr, \sigma^n)$ is the reduced $C^*$-algebra of the twisted groupoid $(\gr, \a^n)$ 
corresponding to the  Fell bundle
$$ L_{\sigma}^n=L_{\sigma}^{\otimes n} \longrightarrow \gr_{\Omega },$$
for all $n\neq 0$, and $C_r^*(\gr, \sigma^0)=C_r^*(\gr_{\Omega})$ by convention. Similar results hold for the maximal $C^*$-algebra. 
\end{remark}

\begin{remark}\label{Moritaextensions}
If we take the twisting $\sigma$ as the $PU(H)$-principal bundle over $\gr$ (as in (1) definiton \ref{HSmorphisms}), then the $C^*$-algebra $C^*(R_{\sigma})$ (maximal or reduced) is well defined up to a canonical strong Morita equivalence. 
Indeed, given $\gr_{\Omega}\stackrel{\sigma}{\longrightarrow}PU(H)$, $\gr_{\Omega'}\stackrel{\sigma'}{\longrightarrow}PU(H)$ two strict morphisms defining the same $PU(H)$-principal bundle over $\gr$, they define two canonically Morita equivalent extensions $R_{\sigma}$ and $R_{\sigma'}$. The induced   strong Morita equivalence between $C^*(R_{\sigma})$ and $C^*(R_{\sigma'})$ respects the gradation of the precedent remark and hence it induces a  strong Morita equivalence between $C^*(\gr, \sigma^n)$ and $C^*(\gr, (\sigma')^n)$ (See proposition 3.3 in \cite{TXL} or \cite{Ren87} for further details).
\end{remark}

\begin{definition}\label{twistedkth} 
Following \cite{TXL}, we define the twisted K-theory of the twisted groupoid $(\gr,\sigma)$ by
\begin{equation}
K^i(\gr,\sigma):=K_{-i}(C_r^*(\gr, \a)).
\end{equation}
In particular if $\sigma$ is trivial we will be using the notation (unless specified otherwise) $K^i(\gr)$ for the respective $K$-theory group of the reduced groupoid $C^*$-algebra.
\end{definition}

\begin{remark}
By the remark \ref{Moritaextensions}, the group $K^i(\gr,\sigma)$ is well defined, up to a canonical isomorphism coming from the respective Morita equivalences.
\end{remark}


\begin{remark}
For  the groupoid  given by  a manifold $M\rightrightarrows M$. A twisting on $M$ can be given by a Dixmier-Douday class on $H^3(M,\ZZ)$. In this event, the twisted K-theory, as we defined it, coincides with  twisted K-theory defined in \cite{ASeg,Kar08}. Indeed the $C^*$-algebra $C^*(M,\sigma)$ is Morita equivalent to the continuous trace $C^*$-algebra defined by the correspondant Dixmier-Douady class (see for instance Theorem 1 in \cite{FMW}).
\end{remark}


\subsubsection{The Analytic index morphism of a twisted groupoid}

Let $(\gr,\sigma)$ be a twisted Lie groupoid with the induced central extension $$ R_{\sigma}\rightarrow \gr_{\Omega } $$
associated to
an open cover $\{\Omega_i\}$ of  $\go =M$.
Let  $(M,\sigma_0)$ be  the twisted groupoid  induced by restriction of $\sigma$  to the unit 
space and  let its central extension be 
$$ R_{\sigma_0}\rightarrow M_{\Omega }.$$ Here $M_{\Omega }$ denotes the covering groupoid
\[
   \bigsqcup_{i, \j}\Omega_{ij}  := \bigsqcup_{i, j} \Omega_i \cap \Omega_j  \rightrightarrows  \bigsqcup_i\Omega_i.
\]

We can consider the inclusion of central extensions:
\begin{equation}\label{inclusionext}
\xymatrix{
R_{\sigma_0}\ar[d]\ar@{^{(}->}[r]&R_{\sigma}\ar[d]\\
M_{\Omega}\ar@{^{(}->}[r]&\gr_{\Omega}\\
}
\end{equation}
as an immersion of Lie groupoids $\iota_{\sigma}: R_{\sigma_0} \to R_{\sigma}$ , which is $S^1$-equivariant and with $R_{\sigma_0}$ amenable. We apply then to this immersion the tangent groupoid construction in subsection 1.3.3 
\begin{equation}\label{defgroupoidext}
R_{\iota_{\sigma}}: = R_{\sigma_0}^{N}\times \{0\} \bigsqcup R_{\sigma}\times (0,1]
\end{equation}
in order to define its index morphism as in (\ref{HSindex}):
\begin{equation}\label{indexR}
\ind_{\iota_{\sigma}}:K_*(C_r^*(R_{\sigma_0}^{N}))\longrightarrow K_*(C_r^*(R_{\sigma})).
\end{equation}

As we remarked before, $R_{\sigma_0}^{N}\rightrightarrows \bigsqcup_i\Omega_i$ is a Lie groupoid with the structures described in  Remark \ref{normalgrpd}. In fact, $R_{\sigma}^{N}$ is also the $S^1$-central extension of  some Lie groupoid . Let us look at this with more detail.
Write  $\sigma=\{(\Omega_i,\sigma_{ij})\}$ as a $PU(H)$-valued cocycle, consider the inclusions 
$\Omega_{ij}\subset \gr_{\Omega_j}^{\Omega_i}$ and their normal bundles $N_{ij} $. We have a Lie groupoid 
$$A_{\sigma}:=\bigsqcup_{i,j}N_{ij}\rightrightarrows \bigsqcup_i\Omega_i,$$
where the source and target maps are the evident ones and where the product is induced from the derivation in the normal direction of the product 
$$\gr_{\Omega_j}^{\Omega_i} \times_{\Omega_j}\gr_{\Omega_k}^{\Omega_j} \longrightarrow \gr_{\Omega_k}^{\Omega_i}.$$

We have also a strict morphism of groupoids 
$$A_{\sigma}\stackrel{\widetilde{\sigma_0}}{\longrightarrow}PU(H)$$
defined  to be equal to $\sigma_0$ in the base direction and constant in the normal vector fibers. The
corresponding central extension can be described as follows.  Consider the Lie algebroid $A\gr\stackrel{\pi}{\longrightarrow}M$ as a Lie groupoid $A\gr \rightrightarrows M$ using its vector bundle structure. Then 
the central extension associated to $(A_\a,  \widetilde{\sigma_0})$ is given by  the following pullback diagram
\begin{equation}\label{twistedA}
\xymatrix{
R_{\sigma_0}^{N}\ar[d]\ar[r]&U(H)\ar[d]\\
(A\gr)_{\Omega}\ar@<.5ex>[d]\ar@<-.5ex>[d]  \ar[r]_{\widetilde{\sigma_0}}&PU(H)\ar@<.5ex>[d]\ar@<-.5ex>[d] \\
\bigsqcup_i\Omega_i \ar[r]  & \{e\},
}
\end{equation}
via the identification $N_{ij}\cong  A\gr |_{\Omega_{ij}}$. In other words the subgroupoid  $R_{\sigma_0}^N$ is the $S^1$-central extension associated to the twisted groupoid $(A\gr\rightrightarrows M, \sigma_0\circ \pi_A)$, where 
$\pi_A:A\gr \rightarrow M$ is considered as a morphism of groupoids. This discussion is of course in agreement with proposition \ref{twist:tangentgoid}. 

Hence we can denote 
\begin{equation}
R_{\sigma_0}^N=R_{\sigma_0\circ \pi_A},
\end{equation}
that is, $R_{\sigma_0}^N$ is the central extension groupoid associated to the twisted groupoid $(A\gr\rightrightarrows M, \sigma_0\circ \pi_A)$.

As in Remark \ref{decomposition}, we have a decomposition
$$C_r^*(R_{\sigma_0\circ \pi_A})=\bigoplus_{n\in\mathbb{Z}}C_r^*(A\gr,(\sigma_0\circ \pi_A)^n).$$
Since, the inclusion of extensions (\ref{inclusionext}) is $S^1$-equivariant, we have that the index morphism (\ref{indexR}) respects the $\ZZ$-gradation
\begin{equation}\label{indexRn}
\ind_{\iota_{\sigma}}:\bigoplus_{n\in\mathbb{Z}}K^*(A\gr,(\sigma_0\circ \pi_A)^n)\longrightarrow 
\bigoplus_{n\in\mathbb{Z}}K^*(\gr,\sigma^n)
\end{equation}


 Consider  the projection   $A^*\gr\stackrel{\pi_{A^*}}{\longrightarrow}M$ and the pullback of $\sigma_0$ to get the twisted groupoid  
\[
(A^*\gr\rightrightarrows A^*\gr,\sigma_0\circ \pi_{A^*}).\]
   Associated to this twisted groupoid we have a central extension given by the pull-back diagram
\begin{equation}\label{twistedA*}
\xymatrix{
R_{\sigma_0\circ \pi_{A^*}}\ar[d]\ar[r]&R_{\sigma_0}\ar[d]\ar[r]& U(H)\ar[d]\\
\bigsqcup A^*\gr|_{\Omega_{ij}}\ar[r]_{\pi_{A^*}}\ar@<.5ex>[d]\ar@<-.5ex>[d] &
\bigsqcup  \Omega_{ij} \ar[r]_{ \sigma_0}\ar@<.5ex>[d]\ar@<-.5ex>[d] &  PU(H)\ar@<.5ex>[d]\ar@<-.5ex>[d] \\
\bigsqcup A^*\gr|_{\Omega_i}\ar[r] & \bigsqcup \Omega_i \ar[r] & \{e\}. 
}
\end{equation}

As in the untwisted case, we have the following proposition.

\begin{proposition}\label{twistedFourier}
The fiberwise Fourier transform gives an isomorphism of $C^*$-algebras
\begin{equation}
C_r^*(R_{\sigma_0\circ \pi_A})\cong  C_r^*(R_{\sigma_0\circ \pi_{A^*}}).
\end{equation}
which preserves the gradation under the $S^1$-action. In particular we have an isomorphism   between the correspondent twisted $K$-theory groups
$$K^*(A\gr ,\sigma_0\circ \pi_{A}) \cong  K^*(A^*\gr,\sigma_0\circ \pi_{A^*}).$$
The right side is the twisted K-theory of the topological space $A^*\gr$ with the pull-ball
twisting $(\pi_{A^*})^*\a_0$.
\end{proposition}

We will prove the proposition above in the more general case of vector bundles. Indeed,
given  a vector bundle $E\stackrel{\pi_E}{\longrightarrow}X$ and a twisting $\beta$ on X, we can consider two twisted groupoids: the first is $(E,\beta\circ \pi_E)$ where $E\rightrightarrows X$ is considered as a groupoid and $\pi$ as a groupoid morphism, and the second is $(E^*,\beta \circ\pi_{E^*})$ where $E^*\rightrightarrows E^*$ is the unit groupoid with the twisting  on the topological space $E^*$. 

\begin{proposition}
The Fourier transform gives an isomorphism 
\begin{equation}\nonumber
C_r^*(R_{\beta\circ \pi_E})\cong  C_r^*(R_{\beta\circ \pi_{E^*}}).
\end{equation}
which preserves the gradation under the $S^1$-action. In particular we have an isomorphism of $C^*$-algebras
\begin{equation}\nonumber
C_r^*(E ,\beta\circ \pi_E) \cong  C_r^*(E,\beta \circ\pi_{E^*}).
\end{equation}
\end{proposition}

\begin{proof}
We will give the explicit isomorphism $C_r^*(R_{\beta\circ \pi_E})\stackrel{\mathscr{F}}{\longrightarrow}
C_r^*(R_{\beta\circ \pi_{E^*}})$. It is defined at the level of $\ci_c(R_{\beta\circ \pi_E})$ as follows:
Let $f\in \ci_c(R_{\beta\circ \pi_E})$, and $[(\eta_{ij},u)]\in R_{\beta\circ \pi_{E^*}}
=\bigsqcup_{i,j}(E^*|_{\Omega_{ij}}\times_{PU(H)}U(H))$ (see diagram (\ref{twistedA*}) above). We let
\begin{equation}
\mathscr{F}(f)([(\eta_{ij},u)])=\int_{E_{\pi(\eta_{ij})}} e^{-i \eta_{ij}(X)}f([(X,u)])dX.
\end{equation}
For $[(\eta_{ij},u)]\in R_{\beta\circ \pi_E}
=\bigsqcup_{i,j}(E^*|_{\Omega_{ij}}\times_{PU(H)}U(H))$ and  $X\in E_{\pi(\eta_{ij})}$,  we have
\[
[(X,u)]\in R_{\beta\circ \pi_E}=
\bigsqcup_{i,j}(E|_{\Omega_{ij}}\times_{PU(H)}U(H))
\]
 by the definition of the morphisms ${\beta\circ \pi_E}$ (\ref{twistedA}) and ${\beta\circ \pi_{E^*}}$ (\ref{twistedA*}). It is also immediate that this function is $S^1$-equivariant. 

Now, as in the untwisted  case, the function $\mathscr{F}(f)$ is not compactly supported. But  it can be shown that it satisfies a Schwartz (rapid decaying) condition in the direction of the vector fibers of $E$ (see for instance proposition 4.5 in \cite{Ca4}). 

Let us first assume $E=X\times \mathbb{R}^q$. We remark that we are not assuming triviality of the twisiting $\beta$ on $X$. In this case we have a very clear description of the extension  groupoids:
First for $E\rightrightarrows X$,
$$R_{\beta\circ \pi_E}=R_{\beta}\times \mathbb{R}^q\rightrightarrows \bigsqcup_{i}\Omega_i,$$
where $\mathbb{R}^q$ is considered as an additive group and $R_{\beta}$ is the extension groupoid associated to $(X,\beta)$.
Then for $E^*\rightrightarrows E^*$,
$$R_{\beta\circ \pi_{E^*}}=R_{\beta}\times \mathbb{R}^q\rightrightarrows \bigsqcup_{i}\Omega_i \times \mathbb{R}^q,$$
where $\mathbb{R}^q\rightrightarrows \mathbb{R}^q$ is considered as an unit groupoid ($\mathbb{R}^q$ as a total space).
Then, under this situation, it is immediate that 
\begin{center}
$C_r^*(R_{\beta\circ \pi_E})=C_r^*(R_{\beta},(C_0(\mathbb{R}^q),*)),$
and
$C_r^*(R_{\beta\circ \pi_{E^*}})=C_r^*(R_{\beta},(C_0(\mathbb{R}^q),\cdot)),$
\end{center}
where in the first $C_0(\mathbb{R}^q)$ we have the convolution product and in the second the point-wise  product.
Since the product on $C_r^*(R_{\beta},(C_0(\mathbb{R}^q),*))$ is given by (at the $\ci_c(R_{\beta},(C_0(\mathbb{R}^q),*))$ level)
$$(f*g)(r)=f(r)*g(r),$$
we obtain, thanks to the continuity of the Fourier transform, that $\mathscr{F}$ gives an isomorphism 
$C_r^*(R_{\beta\circ \pi_E})\cong C_r^*(R_{\beta\circ \pi_{E^*}})$.

For a non-trivial vector bundle $E$, in order to see that $\mathscr{F}$ gives the desired isomorphism from $C_r^*(R_{\beta\circ \pi_E})$ to $C_r^*(R_{\beta\circ \pi_{E^*}})$ we may use trivializating charts of the vector bundle E and consider the correspondent decomposition of $C_r^*(R_{\beta\circ \pi_E})$ as in equation (15) in \cite{Ca4} (see Proposition 4.5 in loc.cit. for more details).

\end{proof}


Now we define the  index analytic morphism for the twisted groupoid $(\gr , \a)$ which agrees with the usual  analytic index morphism for the groupoid $\gr$ when the twisting $\sigma$ is trivial.

\begin{definition}\label{twistedindexmorphism}
Let $(\gr,\sigma)$ be a twisted Lie groupoid. The analytic index morphism for $(\gr,\sigma)$
\begin{equation}
\ind_{a, (\gr,\sigma)}:K^i(A^*\gr,\pi^*\sigma_0)\longrightarrow K^i(\gr,\sigma)
\end{equation}
is defined to be the composition of the Fourier isomorphism from proposition \ref{twistedFourier} and the degree one term of the index morphism $\ind_{\iota_{\sigma}}$ (\ref{indexRn}).
\end{definition}

\subsection{Propreties of the twisted analytic  index morphism}

In this section we will  establish  three properties of  twisted analytic  index morphisms. These are the analogs to the axioms stated in \cite{AS} and will allow us to prove a twisted index theorem for foliations and in general we can then talk of a twisted index theory for these objects.

{\bf Notation.} We recall for the benefit of the reader that $K^*(\gr,\a)$ denotes the K-theory group of the reduced twisted groupoid $C^*$-algebra $K_{-*}(C^*_r(\gr,\a))$ unless specified otherwise (definition \ref{twistedkth} above, see also \cite{TXL}). In particular if $\sigma$ is trivial we will be often using the notation $K^*(\gr)$ for the respective $K$-theory group of the reduced groupoid $C^*$-algebra. Also, when the groupoid is a space, this notation is consistent with the notation for the twisted topological K-theory.

In \cite{TXL} (Proposition 3.7), the Bott periodicity for twisted groupoids was stated, we prove now that the Bott morphism is compatible with the twisted analytic  index morphism.

\begin{proposition}\label{Botttwist}
The twisted analytic  index morphism  $ind_{a,(\gr,\sigma)}$ is compatible with the Bott morphism,
$\it{i.e.}$, the following diagram is commutative
\[
\xymatrix{
K^*(A^*\gr,\pi^*\sigma_0) \ar[rrr]^-{\ind_{a,(\gr,\sigma)}} \ar[d]_-{Bott}
& & & K^*(\gr,\sigma) \ar[d]^-{Bott} \\
K^*(A^*\gr\times,\Rr^2,\pi^*(\sigma\circ p)_0) \ar[rrr]_-{\ind_{a,(\gr\times \Rr^2,\sigma\circ p)}} & & & K^*(\gr\times \Rr^2,\sigma\circ p)
}
\]
where $p:\gr\times \Rr^2\longrightarrow \gr$ is the projection.
\end{proposition}

\begin{proof}
We use the same notations as above, in particular we note 
$R_{\iota_{\sigma}}= R_{\sigma_0}^{N}\times \{0\} \bigsqcup R_{\sigma}\times (0,1]$ the deformation groupoid (as in (\ref{defgroupoidext})) which gives  rise to the twisted index  morphism (\ref{indexRn}) for $(\gr,\sigma)$.  It is
immediate to see   that the corresponding  deformation groupoid for $\sigma\circ p$ is simply $R_{\iota_{\sigma}}\times \mathbb{R}^2$.

Now, the Bott isomorphism considered here is the usual Bott isomorphism for  K-theory of $C^*$-algebras, {\it i.e.}, given by  the product with the  Bott element in $K_0(C_0(\mathbb{R}^2))$. 
The product in K-theory (for $C^*$-algebras) is natural, so in particular it commutes with morphisms $ev_0$ and $ev_1$. Hence, the following diagram is commutative.
\begin{equation}\label{Botttwistdiag}
\xymatrix{
K^*(R_{\sigma_0}^{N})\ar@/^3pc/[rr]^{\ind_{\iota_{\sigma}}}\ar[d]^{\cong}_{Bott}& K^*(R_{\iota_{\sigma}})\ar[r]^{ev_1}\ar[l]_{ev_0}^{\cong}\ar[d]^{\cong}_{Bott}&
K^*(R_{\sigma})\ar[d]^{\cong}_{Bott}\\
K^*(R_{\sigma_0}^{N}\times \mathbb{R}^2)\ar@/_3pc/[rr]_{\ind_{\iota_{\sigma}\times \mathbb{R}^2}}& 
K^*(R_{\iota_{\sigma}}\times \mathbb{R}^2)\ar[r]_{ev_1}\ar[l]^{ev_0}_{\cong}&
K^*(R_{\sigma}\times \mathbb{R}^2)
}
\end{equation}
All the morphisms in the diagram above respect the  $\ZZ$-grading under $S^1$-action, hence the proposition is proved.
\end{proof}

The second property is related with the inclusions of open subgroupoids. Let $\gr \rightrightarrows \go$ be a Lie groupoid and
$\hr \rightrightarrows \ho$ be an open subgroupoid. We have the following compatibility result:

\begin{proposition}\label{opentwist}
Let $\hr\stackrel{j}{\hookrightarrow}\gr$ an inclusion as an open subgroupoid.
The following diagram is commutative:
\[
\xymatrix{
K^*(A^*\hr,\pi^*(\sigma\circ j)_0) \ar[rrr]^{\ind_{a,(\hr,\sigma\circ j)}} \ar[d]_{j_!}
& & & K^*(\hr,\sigma\circ j) \ar[d]^{j_!} \\
K^*(A^*\gr,\pi^*\sigma_0) \ar[rrr]_{\ind_{a,(\gr,\sigma)}} &&  & K^*(\gr,\sigma).
}
\]
where the vertical maps are induced from the inclusions by the open inclusion $j$.
\end{proposition}

\begin{proof}
We use again the notation of the proof of the proposition above. We observe this time that the deformation groupoid $R_{\iota_{\sigma\circ j}}$ which gives the twisted index (\ref{indexRn}) for  $(\hr, \sigma\circ j)$,  is an open subgroupoid of $R_{\iota_{\sigma}}$. The following diagram is obviously commutative.
\begin{equation}\label{opentwistdiag}
\xymatrix{
K^*(R^{N}_{\sigma_0\circ j})\ar@/^3pc/[rr]^{\ind_{\iota_{\sigma \circ j}}}\ar[d]_{j_!}& K^*(R_{\iota_{\sigma\circ j}})\ar[r]^{ev_1}\ar[l]_{ev_0}^{\cong}\ar[d]_{j_!}&
K^*(R_{\sigma \circ j})\ar[d]_{j_!}\\
K^*(R^{N}_{\sigma_0})\ar@/_3pc/[rr]_{\ind_{\iota_{\sigma}}}& K^*(R_{\iota_{\sigma}})\ar[r]_{ev_1}\ar[l]^{ev_0}_{\cong}
&K^*(R_{\sigma})
}
\end{equation}
The proposition easily follows from the diagram above and the fact that all the morphisms in (\ref{opentwistdiag}) respect the $\ZZ$-grading under $S^1$-action.
\end{proof}


The last property (Proposition \ref{Thomtwist} below)    involves in some way the compatibility of the twisted index  morphism   with the product in twisted K-theory. The idea is originated  from Theorem 6.2 in \cite{DLN} and the work  of Connes on tangent groupoids (\cite{Concg}).

\subsubsection{Thom inverse morphism  in twisted K-theory}

Let $N\longrightarrow X$ be a real vector bundle. Consider the fiber product groupoid
$N\times_XN\rightrightarrows N$ over $X$ whose  Lie algebroid is
$N\oplus N\stackrel{\pi_N}{\longrightarrow} N$.  Note that the groupoid $N\times_XN\rightrightarrows N$ is Morita equivalent to the  unit  groupoid $X\rightrightarrows X$, that is,  there is an isomorphism in the Hilsum-Skandalis category 
$$N\times_X N\stackrel{\mathscr{M}}{\longrightarrow}X.$$

 Debord-Lescure-Nistor showed in \cite{DLN} (Theorem 6.2) that, modulo Morita equivalence isomorphism, the  analytic index  morphism  for the groupoid $N\times_XN\rightrightarrows N$
$$
\xymatrix{
K^*(N\oplus N^*)\ar[rr]^{\ind_{a,N\times_XN}} && K^*(N\times_XN),
}
$$ is the inverse of the Thom isomorphism for the 
vector bundle $ N\oplus N\stackrel{\pi_X}{\longrightarrow} X$. In other words the following diagram is commutative
\begin{equation}
\xymatrix{
K^*(N\oplus N^*) \ar[d]_{\ind_{a,N\times_XN}} \ar[r]^-{=}
 &K^*(N\oplus N^*) \ar[d]^-{Thom^{-1}}    \\
K^*(N\times_XN)\ar[r]_-{\mathscr{M}}&  K^*(X).
}
\end{equation}
We will briefly recall the steps of the proof of the above assertion since we are going to follow similar steps in the twisted case.
We recall that analytic indices for groupoids are realized by  deformation groupoids. This point of view is essential for the proof of the above assertion in \cite{DLN}. Indeed, they consider the tangent groupoid of $N\times_XN$, which they called the Thom groupoid, denoted by $\mathscr{T}_N$.
Hence the above diagram takes the following form
\begin{equation}
\xymatrix{
K^*(N\oplus N) \ar[d]_{ev_0^{-1}}  \ar[rr]_-{Fourier}^-\cong
 & &K^*(N\oplus N^*) \ar[dd]^-{Thom^{-1}}_-{\cong}  \\
K^*(\mathscr{T}_N)\ar[d]_{ev_1} &&\\
K^*(N\times_XN)\ar[rr]_-{\mathscr{M}}^-{\cong}& & K^*(X).
}
\end{equation}
They consider first the morphism  $T_0:=Fourier\circ Thom:K^*(X)\longrightarrow K^*(N\oplus N)$, the first pointed arrow below (in terms of KK-elements). Then, the main part of their proof is to  construct the morphism $T: K^*(X)\longrightarrow K^*(\mathscr{T}_N)$ such that its  corresponding  evaluations make   the following  diagram  commute.
\begin{equation}
\xymatrix{
K^*(N\oplus N) \ar[d]_{ev_0^{-1}}  \ar[rr]_-{Fourier}^-\cong
 & &K^*(N\oplus N^*) \ar[dd]^-{Thom^{-1}}_-{\cong}  \\
K^*(\mathscr{T}_N)\ar[d]_{ev_1} &&\\
K^*(N\times_XN)\ar[rr]_-{\mathscr{M}}^-{\cong}& & K^*(X) \ar@{.>}[lluu]_{T_0}\ar@{.>}[llu]^T.
}
\end{equation}

Let us now discuss the twisted case. Given a  twisting  $\beta = \{(\Omega_i, \beta_{ij})\}$ on the space 
$X$ with an open cover $\Omega =\{\Omega_i\}$ of $X$. The bundle $N\oplus N^*$ has  a complex structure, hence we have
 the Thom isomorphism (\cite{CW}, \cite{Kar08}) in twisted K-theory
\begin{equation}\label{twisted:Thom}
\xymatrix{
Thom_\b :  K^i(X,\beta)\ar[r]^-{\cong} 
&K^i(N\oplus N^*,\pi_X^*\beta),
}
\end{equation}
where $N\oplus N^*$ is seen as a total space   and $\pi_X^*\beta$ is a twisting in the classic sense.

We will give now a KK-theoretic description of this isomorphism (modulo Fourier isomorphism ). 
Consider the associated groupoid $N\oplus N\rightrightarrows X$ (using the vector bundle structure) and the corresponding groupoid twisting $\beta\circ \pi_X$, where $\pi_X: N\oplus N\rightarrow X$ is seen as a groupoid morphism.  
Consider the $S^1$-central extension associated to the twisted groupoid $(N\oplus N,\beta\circ \pi_X)$:
\begin{equation}
S^1\longrightarrow R_{N\oplus N,\beta}\longrightarrow (N\oplus N)_{\beta}
\end{equation}
Hence,  
\begin{center}
$(N\oplus N)_{\beta}=\bigsqcup_{i,j}(N\oplus N)_{ij}\rightrightarrows \bigsqcup_i\Omega_i$ and 
$R_{N\oplus N,\beta}=\bigsqcup_{i,j}(N\oplus N)_{ij}\times_{PU(H)}U(H)\rightrightarrows \bigsqcup_i\Omega_i$,
\end{center}
where the pullback is taken by  the groupoid morphism  $(N\oplus N)_{\beta}\stackrel{\beta\circ \pi_X}{\longrightarrow}PU(H)$. 
 Similarly, we have  the  $S^1$-central  extension associated to $(X,\beta)$:
\begin{equation}
S^1\longrightarrow R_{X,\beta}\longrightarrow X_{\beta}, 
\end{equation}
where $X_{\beta}=\bigsqcup_{i,j}\Omega_{ij}\rightrightarrows \bigsqcup_i\Omega_i$ and 
$R_{X,\beta}=\bigsqcup_{i,j}\Omega_{ij}\times_{PU(H)}U(H)\rightrightarrows \bigsqcup_i\Omega_i$.


We are going to describe an element of $KK^*(R_{X,\beta},R_{N\oplus N,\beta})$ which realises the twisted Thom isomorphism (\ref{twisted:Thom}) modulo Fourier isomorphism. 
Let $q:  R_{N\oplus N,\beta}\to X$ be the composition of the obvious maps 
$$R_{N\oplus N,\beta}\rightarrow (N\oplus N)_{\beta}\rightarrow \bigsqcup_i\Omega_i\rightarrow X.$$
Let us now construct the  $(C^*(R_{X,\beta})-C^*(R_{N\oplus N,\beta}))$-Kasparov bimodule $\mathscr{H}_0$. It is the completion of 
$\ci_c(R_{N\oplus N,\beta},q^*\Lambda (N\oplus N)\otimes \Omega^{\frac{1}{2}})$ with respect to its $\ci_c(R_{N\oplus N,\beta},\Omega^{\frac{1}{2}})\subset C^*(R_{N\oplus N,\beta})$ valued inner product:
\begin{center}
$\langle \xi,\eta\rangle(\gamma)=\int_{\gamma_1\cdot \gamma_2=\gamma} \langle \xi(\gamma_1),\eta(\gamma_2))\rangle_{N\oplus N}$, \, \, for $\xi,\eta \in \ci_c(R_{N\oplus N,\beta},q^*\Lambda (N\oplus N)\otimes \Omega^{\frac{1}{2}})$,
\end{center}
where $\langle, \rangle$ denotes the interior product on the (canonical) Hermitian bundle $N\oplus N$, and  the integration is as usual over the resultant 1-density.

The right module structure is given by
\begin{equation}\nonumber
(\xi f)(\gamma)=\int_{\gamma_1\cdot \gamma_2=\gamma} \xi(\gamma_1)f(\gamma_2),
\end{equation}
for $\xi \in \ci_c(R_{N\oplus N,\beta},q^*\Lambda (N\oplus N)\otimes \Omega^{\frac{1}{2}})$ and $f\in \ci_c(R_{N\oplus N,\beta},\Omega^{\frac{1}{2}})$.

The left $\ci_c(R_{X,\beta})$ action on $\ci_c(R_{N\oplus N,\beta},q^*\Lambda (N\oplus N)\otimes \Omega^{\frac{1}{2}})$  is given by
\begin{equation}\nonumber
(g \xi)(\gamma)=g(\pi(\gamma))\cdot \xi(\gamma),   
\end{equation}
for $\xi \in \ci_c(R_{N\oplus N,\beta},q^*\Lambda (N\oplus N)\otimes \Omega^{\frac{1}{2}})$ and $g\in \ci_c(R_{X,\beta})$, where $\pi:\bigsqcup_{ij}(N\oplus N)_{ij}\longrightarrow \bigsqcup_{ij}\Omega_{ij}$ is the canonical projection.

We let $F_0$ to be the endomorphism of $\mathscr{H}_0$, densely defined on $\ci_c(R_{N\oplus N,\beta},
q^*\Lambda (N\oplus N)\otimes \Omega^{\frac{1}{2}})$ by 
\begin{equation}
F_0s(v,w)=\int_{(w',v')\in N_x\times N_x^*}e^{i\,(w-w')\cdot v'}C(v+i v')s(v,w')dw'dv',
\end{equation}
where $C$ denotes the   Clifford  action of $N\oplus N$ on $\Lambda (N\oplus N)$.

\begin{proposition}\label{DLN61}
With the notations above, denote by 
$$
\xymatrix{
T_\b :  K^i(X,\beta)\ar[r] 
&K^i(N\oplus N,\pi_X^*\beta),
}
$$
the degree one morphism (with respect to gradation (\ref{indexRn})) induced by the KK-element $[\mathscr{H}_0,F_0]\in KK^*(R_{X,\beta},R_{N\oplus N,\beta})$. The following diagram commutes
$$
\xymatrix{
K^i(X,\beta)\ar[d]_-{=}\ar[r]^-{T_\b} 
&K^i(N\oplus N,\pi_X^*\beta)\ar[d]^-{Fourier \, \cong}\\
 K^i(X,\beta)\ar[r]_-{Thom_\b}^-{\cong} 
& K^i(N\oplus N^*,\pi_X^*\beta). 
}
$$
In particular the morphism $T_\b$ is invertible.
\end{proposition}

\begin{proof}
When the twisting $\beta$ is trivial, the statement reduces precisely to the one on proposition 6.1 in \cite{DLN}, together with comments and remarks below it. Then, in this case we already have the result. For general case, the claim follows from  the Mayer-Vietoris sequence in twisted K-theory for groupoids in \cite{TXL} proposition 3.9.
\end{proof}

We can state the following proposition which is the twisted analog of Theorem 6.2 in \cite{DLN}. 

\begin{proposition}\label{Thomtwist}
With the above notations we have the following commutative diagram
\begin{equation}
\xymatrix{
K^i(N\oplus N^*,\pi^*\beta) \ar[d]_-{\ind_{a,(N\times_XN, \beta \circ \mathscr{M})}}  \ar[rr]^-{Fourier \ \cong }
  &&  K^i(N\oplus N,\pi^*\beta)  \ar[d]^-{T_\b^{-1}}
   \\
K^i(N\times_XN, \beta \circ \mathscr{M})\ar[rr]_-{Morita \ \cong }& & K^i(X,\beta),
}
\end{equation}
\end{proposition}

\begin{proof}
Consider the deformation groupoid which realizes the twisted index morphism for  $(N\times_XN, \beta \circ \mathscr{M})$: 
$$R_{\beta^{tan}}=R_{N\oplus N, \pi^* \b}\times \{0\}\bigsqcup R_{N\times_X N,\beta \circ \mathscr{M}}\times (0,1].$$
That is, its deformation index gives the element $[ev_0]^{-1}\otimes [ev_1]\in KK(R_{N\oplus N,  \pi^*\b },R_{N\times_X N,\beta \circ \mathscr{M}})$ whose degree one  induces the analytical index morphism $\ind_{a,(N\times_XN, \beta \circ \mathscr{M})}$, modulo the Fourier isomorphism.

We will define a KK-element in $KK(R_{X,\beta},R_{\beta^{tan}})$ in the following way: let $\mathscr{H}$ be the 
$C^*(R_{\beta^{tan}})$-Hilbert module completion of $\ci_c(R_{\beta^{tan}},q^*\Lambda(N\oplus N))$ where $q:R_{\beta^{tan}}\rightarrow X$ is the obvious projection (using the target map of $R_{\beta^{tan}}$ for instance).
The endomorphism of $\mathscr{H}$ is given by
\begin{equation}
Fs(v,w,t)=\int_{(w',v')\in N_x\times N_x^*}e^{\frac{w-w'}{t}\cdot v'}C(v+i v')s(v,w')dw'\frac{dv'}{t^n}.
\end{equation}
for $t\neq 0$ (where $n=dim_{vect}N$) and by $F_0$ for $t=0$. We obtain an element $[\mathscr{H},F]\in KK(R_{X,\beta},R_{\beta^{tan}})$. Denote by $T_{\beta^{tan}}:K(R_{X,\beta})\longrightarrow K(R_{\beta^{tan}})$ the induced morphism in K-theory. It fits by construction in the following commutative diagram:
\begin{equation}
\xymatrix{
&&K(R_{N\oplus N,\beta})\\K(R_{X,\beta})\ar[rur]|{T_{\beta}}\ar[rr]|{T_{\beta^{tan}}}
\ar[rrd]|-{e_1\circ T_{\beta^{tan}}}&&K(R_{\beta^{tan}})\ar[u]_-{e_0}\ar[d]^{e_1} \\&&K(R_{N\times_XN,\beta}) 
}
\end{equation}
To conclude the proof it is enough to show that $\mathscr{M}\circ e_1\circ T_{\beta^{tan}}$ gives the identity in $K(R_{X,\beta})$.  Now, as a KK-element, 
$T_{\beta^{tan}}\otimes e_1\otimes \mathscr{M}\in KK(R_{X,\beta},R_{X,\beta})$ can be represented by $(H_{\Lambda E},F_1)$ where 
$$H_{\Lambda E}=(L^2(\bigsqcup_{\{i\in I, x\in \Omega_i\}}N_x\times S^1, \Lambda E))_{x\in \bigsqcup \Omega_j},$$
and the vector bundle is the pullback of $\Lambda E$ by the canonical projection 
$\bigsqcup_{\{i\in I, x\in \Omega_i\}}N_x\times S^1\rightarrow \bigsqcup_i\Omega_i\rightarrow X$.
The operator $F_1$ is as the operator F evaluated at $t=1$ but identified with a continuous family of Fredholm operators acting on $L^2(\bigsqcup_{\{i\in I, x\in \Omega_i\}}N_x\times S^1, \Lambda E)$ by
$$F_1s(x,v)=\int_{(w',v')\in N_x\times N_x^*}e^{i\, (v-w')\cdot v'}C(v+i v')s(x,w')dw'dv'.$$
The lemma 2.4 in \cite{CS} applies here and we have then that $(H_{\Lambda E},F_1)$ defines the unit element of $KK(R_{X,\beta},R_{X,\beta})$.
\end{proof}

\section{Longitudinal twisted index theorem for foliations}

Consider the case of a regular foliation $(M,F)$   with a twisting $\sigma$ on its leaf space $M/F$ in the sense of Definition \ref{twistingleaves}, 
by using the holonomy groupoid $\gr$ of $(M, F)$. The induced twisting on  $M$ is denoted by $\sigma_0$.  Recall that the Lie algebroid  $A\mathscr{G}$  is $F\rightarrow  M$.  In this section, we  define  the topological index for  $(M/F,  \sigma)$ and show
that it agrees with the twisted analytical index morphism for $(M/F,  \sigma)$. 

\subsection{Twisted topological index for foliation} \

Let $i :M\hookrightarrow \mathbb{R}^{2m}$ be an embedding and $T$ be the total space of the normal vector bundle $\pi_T: T\to M$  to the foliation in $\mathbb{R}^{2m}$ with $T_x =(i_*(F_x))^{\bot}$ for $x\in M$.  Consider the foliation  $M\times \mathbb{R}^{2m}$ given by  the bundle  $\tilde{F} = F\times \RR^{2m} \to M\times \RR^{2m}$. 
This foliation has $\tilde{\mathscr{G}}=\mathscr{G} \times \mathbb{R}^{2m} \rightrightarrows M\times \RR^{2m}$ as its holonomy groupoid
equipped  with the pull-back twisiting, still denoted by $\sigma$. 
The map  $(x,\xi) \mapsto (x, i(x)+\xi)$   identifies  an open neighborhood of the 0-section of $T$ with an open transversal of  $(M\times \mathbb{R}^{2m},\tilde{F})$, that we still denote by $T$ with the projection
$\pi_T: T\to M$. 
Denote by $N$  the total space of the normal vector bundle to the inclusion $T\subset M\times \mathbb{R}^{2m}$. We  
can identify $N$ with a small open neighborhood  $U$ of $T$ in $M\times \mathbb{R}^{2m}$.

We use the following diagram to streamline the above notations
\[
\xymatrix{
N\oplus N = N\times_T N  \ar[r]^-{\pi_1} \ar[dr]_{\pi_N} & N \ar[rd]^{p_N} \ar[r]^{\cong} &U \ar@{^{(}->}[dr]&  \tilde{F} \cong F\times \RR^{2m}\ar[d]  
& \tilde{\gr} \ar@<.5ex>[dl]\ar@<-.5ex>[dl]  \ar[r]^{\a\circ p}& PU(H) \ar@<.5ex>[dl]\ar@<-.5ex>[dl]  \\
&F \ar[rd]_{\pi_F}  & T \ar@{^{(}->}[r] \ar[d]^{\pi_T}& M\times \RR^{2m} \ar[ld]_{p} \ar[r]&\{e\} \\
& & M \ar@{-->}[r]^{\a_0}& PU(H)&
}, \]
where $\pi_1$ denotes the projection to the first component, and $\a_0$ denotes the restriction of the twisting $\sigma$ on $\gr$ to the unit space $M$.  As $T$ is an open transversal to the foliation $\tilde{F}$ on $M\times \RR^{2m}$, its normal bundle
$N$ is the pull-back  of $\tilde{F} = \pi_T^*F$ to $T$. We can see that $N\cong \pi_T^* F$ as   vector bundles over $T$.
Under the identification $N$ and $U$, $N\times_T N  = N\oplus   N$ is the total space of $\tilde{F}|_U$. Hence, there is a canonically defined projection $\pi_N$ making the above diagram
commutative.  The same arguments imply that  the total space $N\oplus N^*$ (as a vector bundle over $T$) is diffeomorphic to the total space $F^*\times \RR^{2m}$ (as a vector bundle over $M$).

\begin{lemma} \begin{enumerate}
  
\item There exists an open  neighborhood
$U$ of $T$ in $M\times \mathbb{R}^{2m}$ such  that
the honolomy groupoid of  $\tilde{F}|_U$ over $U$  is strictly isomorphic to  the  groupoid 
$  N\times_T N \rightrightarrows N$ associated to the submersion  $N\to T$.   The latter groupoid
$  N\times_T N \rightrightarrows N$ is Morita equivalent to the groupoid $T\rightrightarrows T$.


\item  Let    $\mu: (N\times_TN\rightrightarrows N)  \rightarrow (T\rightrightarrows T)$ be the Morita equivalence isomorphism of Lie groupoids, $j:N\times_T N \cong \tilde{\gr}|_U \hookrightarrow \gr\times \mathbb{R}^{2m}$ be  the inclusion as an open subgroupoid, and 
$p:\tilde \gr\times \mathbb{R}^{2m}\longrightarrow \gr$ be  the projection. Then
the twisting  $\sigma_0\circ \pi_T\circ \mu$ agrees  with  the twisting $ \sigma\circ p\circ j $ over the  groupoid
$N\times_T N \rightrightarrows N$.
 \end{enumerate}
\end{lemma}

\begin{proof} The proof follows from the above commutative diagram and the fact   that  $T$ is  an open transversal to the foliation $\tilde{F}$ on $M\times \RR^{2m}$.
\end{proof}

This lemma enables us to define the twisted  topological index of $(M,F,\sigma)$ which agrees with
the Connes-Skadalis' longitudinal topological index of $( M ,F)$ when the twisting is trivial. 

\begin{definition}\label{twistedtopindex}  By the {\bf  twisted topological index of $(M,F,\sigma)$} we mean the morphism
\[
\ind_{t,(M/F,\sigma)}:K^i(F^*, \sigma_0 \circ \pi_F)\longrightarrow K^i(M/F,\sigma):= K^i (\gr,\sigma)
\]
given by the composition of various isomorphisms in  twisted K-theory of topological spaces, functoriality of open embedding and Bott isomorphism in  twisted K-theory of groupoids
\begin{equation}\small{
\xymatrix{
  K^i(F^*, \sigma_0 \circ \pi_F) \ar[r]^-{Bott}_-{\cong }& 
  K^i(F^*\times \RR^{2m}, (\sigma\circ p)_0 \circ \pi_F) \ar[r]^-{\cong}    & K^i( N\oplus N^*, \sigma_0\circ \pi_F\circ \pi_N)) \ar[d]^-{Thom^{-1}}_-{\cong} &  \\ 
     K^{i}(N\times_T N,\sigma\circ p\circ j) 
\ar[d]^\cong &K^{i}(N\times_T N,\sigma_0 \circ \pi_T\circ \mu ) \ar[l]^{\cong} &K^i(T,\a_0\circ \pi_T)\ar[l]^-{Morita}_-{\cong}&  \\ K^{i}(\tilde{\gr}|_U,\sigma\circ p\circ j)
\ar[r]^-{j_!} &
 K^{i}(\gr\times \Rr^{2m},\sigma\circ p)
\ar[r]^-{Bott^{-1}}_-\cong&K^{i}(\gr ,\sigma )  .
}}
\end{equation}
 
\end{definition}

 \subsection{The twisted longitudinal index theorem} \
 
 For a regular foliation $(M,F)$ with  a twisting $\sigma$ the space of leaves, that is, a generalized morphism
 from the honolomy groupoid $M/F$ to $PU(H)$, we have  the  twisted analytic index morphism 
 (Definition \ref{twistedindexmorphism})
\[
\ind_{a,(M/F,\sigma)}: K^i(F^*, \sigma_0 \circ \pi_F) \longrightarrow K^i(M/F,\sigma).
\]
The main theorem of the paper is to establish the equality between the twisted analytic index  
and the twisted topological index for $(M/F, \a)$ 
which generalizes the  longitudinal index theorem of Connes and Skadalis (\cite{CS}) for the trivial twisting. 

\begin{theorem}\label{twistlongind}
For a regular foliation $(M,F)$ with a twisting $\xymatrix{\sigma:M/F \ar@{-->}[r] &PU(H)}$ we have the following equality of morphisms:
\begin{equation}
\ind_{a,(M/F,\sigma)}=\ind_{t,(M/F,\sigma)}: K^i(F^*, \sigma_0 \circ \pi_F) \longrightarrow K^i(M/F,\sigma).
\end{equation}
\end{theorem}

\begin{proof}
We use the same notation and terminology of the definition of the twisted topological index for $M/F, \a)$.
Notice that  the analytic twisted index morphism  is compatible with the Bott isomorphism (Proposition \ref{Botttwist}) 
\[
\xymatrix{ K^i(F^*, \sigma_0 \circ \pi_F) \ar[r]^-{Bott}_-{\cong }  \ar[d]_{\ind_{a,(\gr,\sigma)}} & 
  K^i(F^*\times \RR^{2m},(\sigma\circ p)_0\circ \pi_F) \ar[d]^{\ind_{a,(\tilde\gr,\sigma\circ p)}}   \\
  K^i(M/F,\sigma)\ar[r]_-{Bott}^-\cong  &    K^i(   \gr \times \RR^{2m} ,\sigma\circ p) 
}\]
 and  open  inclusions of subgroupoids (Proposition  \ref{opentwist})
 \[
 \xymatrix{ 
 K^i( A^*\tilde{\gr}|_U,  (\a\circ p\circ j)_0 \circ \pi)   \ar[r]^-\cong \ar[d]_{\ind_{a,(\tilde\gr|_U ,\sigma\circ p)}}  & K^i(F^*\times \RR^{2m}, (\sigma\circ p)_0\circ \pi_F) \ar[d]^{\ind_{a,(\tilde\gr,\sigma\circ p)}}   \\
K^i(   \tilde\gr |_U,\sigma\circ p\circ j) \ar[r]_-{j_!}   & K^i(   \gr \times \RR^{2m} ,\sigma\circ p) .
}\]
 We only need to check that the
analytic index morphism for $(U, \tilde{F}|_U)$
\[
\ind_{a, (\tilde \gr|_U, \a\circ p\circ j)}:   K^i( A^*\tilde \gr|_U,  (\a\circ p\circ j)_0 \circ \pi)  
\longrightarrow  K^{i}(\tilde \gr|_U,\sigma\circ p\circ j)
\]
agrees with the  composition of the following part in the definition of  the twisted topological index for $(M/F, \a)$ 
\[
\xymatrix{
   K^i( N\oplus N^*, \sigma_0\circ \pi_F\circ \pi_N)) \ar[r]^-{Thom^{-1}}_-{\cong} & K^i(T,\a_0\circ \pi_T)\ar[r]^-{Morita}_-{\cong}  & K^{i}(\tilde{\gr}|_U,\sigma\circ p\circ j).
}
\]
Note that $A^*\gr|_U \cong N\oplus N^*$ as vector bundles over $U\cong N$. That is, the twisted analytic  index morphism (modulo the Morita equivalence isomorphism)
is the Thom inverse for  twisted K-theory  of Lie groupoids  as in Proposition \ref{Thomtwist}.  
Hence, we finish the proof of the  twisted longitudinal index theorem. 
\end{proof}

What we just proved in the last theorem is that any index morphism for foliations with twistings satisfying the three properties  in Propositions \ref{Botttwist}  \ref{opentwist} \ref{Thomtwist}  is equal to the twisted topological index. 

\smallskip
Indeed, if ${\it IND}_{(\hr,\beta)}$ is any such index morphism for twisted groupoids, the three mentioned properties imply that the following three square diagrams are commutative precisely by the same arguments as the proof above
{\tiny\[ \xymatrix{ K^*(F^*,\sigma_0\circ\pi_F) 
\ar[r]^-{Bott}_-{\approx} \ar[dd]|-{{\it IND}_{(M/F,\sigma)}} &K^*(F^*\times \RR^{2m},(\sigma\circ p)_0\circ \pi_F) \ar[dd]|-{{\it IND}_{(M/F \times \RR^{2m},\sigma\circ p)}}& K^*(A^*\tilde \gr|_U,  A^*\tilde{\gr}|_U,  (\a\circ p\circ j)_0 \circ \pi)\ar[dd]|-{{\it IND}_{(\tilde \gr|_U,\sigma\circ p\circ j)}} \ar[l]_-{\approx}_-{}&  K^*(N\oplus N^*,  \sigma_0\circ \pi_F\circ \pi_N)\ar[dd]|-{Thom^{-1}}\ar[l]_-{\approx}\\ && &\\ K^*(M/F,\sigma)\ar[r]_-{Bott}^-{\approx}&K^*(M/F\times \Rr^{2m},\sigma\circ p)&K^*(\tilde \gr|_U,\sigma\circ p\circ j)\ar[l]^-{j!}&K^*(T,\a_0\circ \pi_T)\ar[l]_-{\approx}^-{Morita} } ,
	\]
}
Hence, the big diagram is also commutative, and we obtain the equality ${\it IND}_{(M/F,\sigma)}=\ind_{t,(M/F,\sigma)}$.

\begin{remark} \begin{enumerate}
\item There is a second definition of
the twisted analytic index using projective pseudo-differential operators along the leaves which also satisfies
the three properties  in Propositions \ref{Botttwist}  \ref{opentwist} \ref{Thomtwist}.  We will return to this issue in a separate paper. 
 
\item The proof of the longitudinal index theorem we propose is not exactly that of Connes-Skandalis (\cite{CS})  but more in the spirit of the index theorem proved in \cite{Ca4}, theorem 6.4, which is based on the fact that the Thom (inverse) isomorphism can be realized as the index of some deformation groupoid, \cite{DLN}.

\end{enumerate}

\end{remark}

\section{ Wrong way functoriality}

 In this section we construct a push-forward map  in twisted K-theory for any smooth oriented map from 
 a manifold to $M/F$,   the space of leaves of a foliation $(M, F)$.  This construction generalizes 
Connes-Skandalis' push-forward map  in  K-theory for any K-oriented map from a manifold to $M/F$. The main result will be the functoriality of these wrong way morphisms.

Throughout  this section, we assume that   with respect to an open cover  $\{U_\alpha\}$  of $M$ and the  foliation  charts 
\[
 k_\alpha:  U_\alpha \cong 
  \RR^p \times \RR^q \to  \RR^{q} , 
  \]
    the twisting
$\sigma$ on $M/F$ is given by a $PU(H)$-valued 1-cocycle  
\[
\sigma_{\alpha\beta}:  \gr_{U_\alpha}^{U_\beta} \longrightarrow PU(H).
\]


Recall that  in  \cite{CS} a smooth map $f: W\to M/F$ is given by a principal right $\gr_M$-bundle over $W$
\[
\xymatrix{
 &  \gr_f \ar[dl]^{r_f}   \ar[rd]_{s_f}&\gr_M \ar@<.5ex>[d]_{r\ } \ar@<-.5ex>[d]^{\ s}  \\
W & &  M , 
}
\]
and $f: W\to M/F$
is called a submersion if the map $s_f$ is a submersion.  Equivalently,  a smooth map $f: W\to M/F$ is given by
a $\gr_M$-valued 1-cocycle  $(\Omega_i, f_{ij})$
\begin{equation}\label{fcocycle}
f_{ij}: \Omega_i\cap \Omega_j \longrightarrow \gr_M, 
\end{equation}
with respect to an open covering $\{\Omega_i\}$ of $W$ such that 
\begin{equation}\nonumber
f_{ij} (x) \circ f_{jk}(x) = f_{ik} (x)
\end{equation}
for any $x\in \Omega_i \cap \Omega_j \cap \Omega_k$.  Note that   $f_{ii}:  \Omega_i \longrightarrow M$  due to the fact that  $ f_{ii} (x) \circ f_{ii}(x) = f_{ii} (x)$ for any $x\in \Omega_i$.    We always assume that for each $\Omega_i$
\[
\overline{f_{ii} (\Omega_i)} \subset U_{\alpha(i)}
\]
for a chosen foliation chart $U_{\alpha (i)}$.  In particular, the pull-back twisting of $\sigma$ by $f$ on $W$, denoted by
$f^*\sigma$,  is
the composition 
\[
f^*\sigma =  \sigma_{\alpha (i)\alpha (j)} \circ f_{ij}:  \Omega_i\cap \Omega_j \longrightarrow  
 \gr_{U_{\alpha(i)}}^{U_{\alpha (j)}} \longrightarrow PU(H).
 \]

Let $\nu_F$ be the transverse bundle to the foliation $F$, that is, $(\nu_F)_x = T_xM/F_x$.
  One can check that the local
 vector bundles $f^*_{ii}\nu_F$ can be glued together using    $f_{ij}$  to form a real vector bundle 
 over $W$. We denote this vector bundle by $f^*\nu_F$.  A map $f: W\to M/F$ is called
 oriented if $TW\oplus f^*\nu_F$ is an oriented vector bundle over $W$.   The orientation twisting of 
 $TW\oplus f^*\nu_F$
 \[\xymatrix{
 o_{TW\oplus f^*\nu_F}:  W  \ar@{-->}[r] & PU(H)}
 \]
 is equivalent to a trivial twisting if and only if $TW\oplus f^*\nu_F$ is K-orientable.
 
We are now going to construct the push-forward map, with a possible degree shift given by $d(f) = dim (W) + rank (\nu_F) \mod 2$, 
\begin{equation}\label{push-forward}
f_!:  K^*(W, f^*\a + o_{TW \oplus  f^*\nu_F}) \longrightarrow K^{*+ d(f)}(M/F, \a)
\end{equation}
associated to any smooth oriented  map 
$f: W\longrightarrow M/F$ and a twisting $\xymatrix{\sigma:M/F \ar@{-->}[r] & PU(H)}$, where $f^*\a$ is the
 pull-back twisting on $W$, and $o_{TX \oplus f^*\nu_F}$ is the orientation twisting of  $TW\oplus \nu_F$. 
   Let 
 \[\xymatrix{  W\ar[rr]^-{f}   \ar[dr]_-{j} & & M/F \\
  & Z  \ar[ru]_-{g} }
 \]
 be a factorization of $f$ such that $j$ is oriented and proper, and $g$ is a submersion. Such a factorization can be
 found in \cite{Cosf} \cite{BH} using the foliation microbundle associated to a Haefliger structure on $W$.   
 According to \cite{CW},
 there is a push-forward map in twisted K-theory for any proper map $j: W  \to Z $
 \[
 j_!:  
  K^*(W, f^*\a + o_{TW \oplus  f^*\nu_F}) \longrightarrow  K^*( Z,g^*\a + o_{TZ \oplus g^*\nu_F}).
  \]
 Therefore, we only need to establish a push-forward map for any submersion $g:  Z   \to M/F$
  \[
g_!:   K^*(  Z, \hat f^*\a + o_{T Z  \oplus   g ^*\nu_F}) \longrightarrow  K^*(M/F, \a ), 
  \]
 and show that
  $f_! =g_! \circ j_! $ doesn't depend on the choice of the factorization $f = g\circ j$.


Let  $f: W\to M/F$ be a submersion,  that is, with respect to an open cover $\{\Omega _i\}$ of $W$, there
is a $\gr_M$-valued 1-cocylce $\{(\Omega_i, f_{ij}\}$ 
\[
f_{ij}:  \Omega_i\cap \Omega_j \longrightarrow \gr_M
\]
 such that  $f_{ii}: \Omega_i \to M$ are transverse to the foliation $F$. 
Denote by   $F_W=f^*(F)$  the pull-back  foliation  on $W$ 
 given by the kernel of the homomorphism (independent of $i$ and integrable)
\[
\pi \circ df_{ii}:  TW|_{\Omega_i} \longrightarrow TM \longrightarrow TM/F, 
\]
with the induced foliation charts given by
$
k_{\alpha(i)} \circ f_{ii}:  \Omega_i \longrightarrow U_{\alpha(i)} \longrightarrow  \RR^q.
$
Note that the transverse bundle to the foliation $F_W$, denoted by $\nu_{F_W}$, is isomorphic to $f^*\nu_F$, 
the pull-back transverse bundle.

Denote by $\gr_W\rightrightarrows W$ the holonomy groupoid of $(W,F_W)$. In \cite{CS} Lemma 4.2, Connes-Skandalis give a explicit (left) action of $\gr_W$ on $\gr_f$. In terms of Hilsum-Skandalis morphisms (Definition 1.1 {\it (i)} in \cite{HS}) this means precisely that  there is a generalized morphism 
$\tilde{f}:\gr_W\longrightarrow \gr_M$ given by  the graph $\gr_f$ seen as a 
$\gr_M-$principal bundle over $\gr_W$
\[
\xymatrix{
\gr_W \ar@<.5ex>[d]\ar@<-.5ex>[d]&\gr_f \ar[ld] \ar[rd]&\gr_M \ar@<.5ex>[d]\ar@<-.5ex>[d]\\
W&&M,
}
\]
or equivalently that we can factorize the submersion $f: W \to M/F$ as follows
\begin{equation}\label{fsubmersion}
\xymatrix{
W\ar[rd]_{p_W}\ar[rr]^f&&M/F\\
&W/F_W\ar[ru]_{\tilde{f}}&
}
\end{equation}
where $p_W:W\longrightarrow W/F_W$ is the natural submersion given by the inclusion of $W$ into
 the unit space $W/F_W$.
 Denote by $\tilde{f}^*\a$ be the pull-back twisting
on $W/F_W$,  then $f^*\a$  is equivalent to
the restriction of $\tilde{f}^*\a$ to the unit space of $\gr_W$. Note that twisting $o_{TW \oplus f^*\nu_F}$
is equivalent to the twisting $o_{\nu_{F_W}}$ as
\[
TW \oplus f^*\nu_F \cong F_W \oplus \nu_{F_W} \oplus f^* \nu_F \cong F_W \oplus  f^*(\nu_F\oplus   \nu_F ),
\]
and $\nu_F\oplus  \nu_F$ has a canonical Spin$^c$ structure. Hence, we have the following Thom isomorphism (Cf. \cite{CW})
\[
K^*(W, f^*\a + o_{TW \oplus \nu_{F_W}}) \cong K^* (W, f^*\a + o_{F_W}).
\]

\begin{definition}\label{p_W} 
The push-forward map for the submersion $p_W: W \to W/F_W$ is defined to the composition 
of the Thom isomorphism in twisted K-theory and the twisted analytic index of $(W/F_W, \tilde{f}^*\a)$
\[ \xymatrix{
(p_W)_!:    K^* (W, f^*\a + o_{F_W}) \ar[r]^-{Thom}_-\cong
&  K^*(F_W^*, \pi_{F_W}^* \tilde f^*\a)  \ar[rrr]^{\ind_{a, (W/F_W, \tilde f^*\a)} }& & & K^*(W/F_W, \tilde f^*\a  ) },
\]
with the degree shifted by the rank of $F_W$. Here $\pi_F$ is the projection $F^*_W \to W$.
\end{definition}

 

Now we are going   to construct a  push-forward map  
\begin{equation}
\tilde{f}_! :K^*(W/F_W,\tilde{f}^*\sigma)\longrightarrow K^*(M/F,\sigma).
\end{equation}
associated to the morphism $\tilde{f}:W/F_W\longrightarrow M/F$.

Note  that a generalized morphism   $\tilde{f}:W/F_W\longrightarrow M/F$ can be equivalently described  (Definition-Proposition 1.1 in \cite{HS}) by a strict morphism of groupoids 
\begin{equation}
\tilde{f}_T:(\gr_W)_{T }\longrightarrow (\gr_M)_{T }
\end{equation}
between the \'etale groupoids obtained from the restriction  to some complete faithful transversals. 
 The way of relating the morphism $\tilde{f}$ with $\tilde{f}_T$  is by means of the Morita equivalence between an holonomy groupoid and  the \'etale groupoid obtained by restriction to a transversal. In our case we have a commutative diagram
\begin{equation}\label{transversgrpds}
\xymatrix{
\gr_W\ar[r]^-{\tilde{f}}&\gr_M\\
(\gr_W)_{T }\ar[r]_-{\tilde{f}_T}\ar[u]^{\mathscr{M}_W}_{\sim}& (\gr_M)_{T }\ar[u]_{\mathscr{M}_M}^{{\sim}},
}
\end{equation}
where $\mathscr{M}$ stands for the corresponding Morita equivalence. 
We briefly recall the construction of $\tilde{f}_T:(\gr_W)_{T}\to (\gr_M)_{T}$ ( Details can be found in \cite{HS}, specially in Definition-Proposition 1.1  and paragraph 3.13).

Recall that  the foliation $F_W$ is constructed locally from the foliation chart
\begin{equation}\label{ki}
k_i:  \Omega_i\stackrel{f_{ii}}{\longrightarrow}U_{\alpha(i)}\stackrel{k_{\alpha (i)}}{\longrightarrow}\mathbb{R}^q
\end{equation}
 where $\{k_{\alpha (i)}: U_{\alpha(i)} \to \RR^q \}$ is the foliation chart  defining the foliation $ F$ on $M$. 
 Let $\T_i$ be a transversal of the foliation $F_W$  on $\Omega_i$ defined by choosing a right smooth inverse to $k_i$.  For each  foliation  chart 
  $k_\alpha: U_\alpha \to \RR^q$,  we can choose a smooth connected transversal 
  \[
  k_\alpha^{-1}:  \RR^q \longrightarrow \T_\alpha \subset U_\alpha, 
  \]
  in particular, $k_\alpha \circ k_\alpha^{-1} = Id$.  We can assume that $\{\T_\alpha\}$ are 
  pairwise disjoint and  $f_{ii}$ defines an embedding 
 \begin{equation}\label{fi}
 f_i: \T_i \longrightarrow \T_{\alpha(i)}
 \end{equation}
 where $\T_{\alpha(i)}$ is a transversal of the foliation  $F_M$ on $U_{\alpha(i)}$.  
Let   $\T_W   =  \bigsqcup_i   \T_i 
$ and $\T_M = \bigsqcup_{i} \T_{\alpha(i)}$. We have then the restricted groupoids
\smallskip
\begin{center}
$ (\gr_W)_{T }:=\bigsqcup_{i,j}(\gr_W)_{\T_j}^{\T_i}\rightrightarrows \T_W$ \, and  \, 
$ (\gr_M)_{T }:=\bigsqcup_{\alpha, \beta}(\gr_M)_{\T_{\alpha}}^{\T_{\beta }}\rightrightarrows \T_M$.
\end{center}

Recall that by definition $\tilde{f}$ is given by a cocycle $\{(\Omega_i, \tilde{f}_{ij})\}_i$ on $\gr_W$ with values on 
$\gr_M$.
The strict morphism of groupoids
\[
\xymatrix{
(\gr_W)_{T }\ar@<+2pt>[d]\ar@<-2pt>[d]\ar[r]^{\tilde{f}_T}& (\gr_M)_{T }\ar@<+2pt>[d]\ar@<-2pt>[d]\\
\T_W\ar[r]_-{\tilde{f}_{T}^{(0)}}&\T_M
}
\]
is defined as  the restrictions of the $\tilde{f}_{ij}$ to the $(\gr_W)_{\T_j}^{\T_i}$, {\it i.e.}, 
\[
\xymatrix{
(\gr_W)_{T_W}=\bigsqcup_{i,j}(\gr_W)_{\T_j}^{\T_i}\ar[rr]^-{\bigsqcup_{i,j}\tilde{f}_{ij}} &&(\gr_M)_{T_M}\subset \gr_M
}
\]
Notice that the above map has its image in $(\gr_M)_{T}$ by construction, where we are canonically identifying $(\gr_M)_{T}$ as the sub-groupoid of $\gr_M$ 
  given by
  \[
 \{ \gamma \in \gr_M: s(\gamma) \in \T_{\alpha(i)} , r(\gamma) \in \T_{\alpha(j)}.\}
 \]

Remark that $\T_W$ is a complete transversal and that, by enlarging $\T_M$, we can also assume $T_M$ is complete. Even more, in our case, the embedding $\tilde{f}_{T}^{(0)}:\T_W\hookrightarrow \T_M$ is \'etale (even we can assume it is a proper open inclusion, see (\ref{ki}), (\ref{fi}) or remark 1.4 in \cite{HS} for more details). In particular $\dim\, \T_W=\dim\, \T_M$.
The groupoids $(\gr_W)_{T}\rightrightarrows \T_W$ and $(\gr_M)_{T}\rightrightarrows \T_M$ are \'etale groupoids. Hence, the two arrow manifolds have the same dimension too: 
$\dim\,(\gr_W)_{T}=dim\, \T_W=\dim\, \T_M=\dim\,(\gr_M)_{T}$. 

What we just have argued can be stated in the following known fact:

\begin{lemma}(\cite{Haef62} p.378,\cite{CS} section IV,\cite{HS})
The homomorphism $\tilde{f}_T:(\gr_W)_{T_W}\longrightarrow (\gr_M)_{T_M}$ is an injective \'etale map. 
\end{lemma}

Now, given  a twisting $\sigma:\gr_M--->PU(H)$,  by taking the refinement of the foliation charts $\{U_\alpha\}$ if necessary, we can construct as above a strict morphism of groupoids
\[
\sigma_T :   (\gr_M)_{T } \longrightarrow PU(H),\]
 which defines a twisting   on $(\gr_M)_{T }$.
 The $C^*$-algebras of two Morita equivalent extensions are also Morita equivalent, \cite{Ren87,TXL}. Hence, we have
\[
K^*(M/F,\sigma)\cong K^*((\gr_M)_T,\sigma_T).
\]


 The twisting $\a_T$ and $ \tilde{f}_T^* \a_T = \a_T \circ \tilde{f}_T$ induce the following injective immersions of  groupoid extensions
\begin{equation}
\xymatrix{
R_W\ar[r]\ar[d]&R_M\ar[d]   \\
(\gr_W)_T\ar[r]_{\tilde{f}_T}&(\gr_M)_T  . \\
}
\end{equation}
Taking  the degree one index morphism associated to the immersion $R_W\hookrightarrow R_M$ as in (\ref{indexR}), we  obtain a homomorphism 
\begin{equation}
(\tilde{f}_T)_! :K^*((\gr_W)_{T },\tilde{f}^*\sigma_T)\longrightarrow K^*((\gr_M)_{T },\sigma_T)
\end{equation}

We have the following lemma.

\begin{lemma}\label{tilde{f}}
With the notations above,  the push-forward map  
\begin{equation}
\tilde{f}_! :K^*(W/F_W,\tilde{f}^*\sigma)\longrightarrow K^*(M/F,\sigma)
\end{equation}
associated to $\tilde{f}: W/F_W \to M/F$ 
given by   the composition
$$K^*(W/F_W,\tilde{f}^*\sigma)\stackrel{\mathscr{M}_W}{\longrightarrow}K^*((\gr_W)_{T },\tilde{f}^*\sigma_T)\stackrel{(\tilde{f}_T)_!}{\longrightarrow}K^*((\gr_M)_{T },\sigma_T)\stackrel{\mathscr{M}_M }{\longrightarrow}
K^*(M/F,\sigma),$$ 
where  $\mathscr{M}_W$ and $\mathscr{M}_M$ denote   
the induced isomorphisms  from the Morita equivalences (\ref{transversgrpds}), does not depend on the choices of complete transversals.
\end{lemma}
\begin{proof}
It is a direct adaptation of the proof of lemma 3.14 in \cite{HS}. Just observe that we can also apply remark 3.12 in \cite{HS}  in our context since the morphism $\tilde{f}:W/F_W\longrightarrow M/F$ is \'etale. 
\end{proof}

The following two propositions will allow us to use the factorization (\ref{fsubmersion}) and the precedent discussion to define the twisted pushforward map for any submersion.

\begin{proposition}\label{CS45}
Let $f:W\longrightarrow M/F$ be a smooth submersion and $\xymatrix{\sigma:M/F \ar@{-->}[r] &PU(H)}$ be a twisting. Let $X\stackrel{g}{\longrightarrow} W/F_W$ be another submersion. 
Let $\tilde{f} : W/F_W \to M/F$ and $ \tilde{g}: X/F_X \to W/F_W$ be the associated generalized morphisms.  Then 
$\tilde{f}_! \circ\tilde{g}_! =(\tilde{f} \circ\tilde{g})_!$.
\end{proposition}
\begin{proof}
By functoriality of the pullback foliation construction (for submersions) we have that 
$F_X =g^*F_W=(\widetilde{f}\circ g)^*F_M$. 
The commutative diagram of immersions of groupoids 
\[
\xymatrix{
(\gr_X)_T \ar[rr]^-{\tilde{f}_T\circ \tilde{g}_T} \ar[dr]_-{\tilde{g}_T} && (\gr_M)_T  \\
& (\gr_W)_T \ar[ru]_-{\tilde{f}_T} }
\]
induces the commutative diagram of immersions of  central extensions
\[
\xymatrix{
R_X \ar[rr]^{\tilde{f}_-T\circ \tilde{g}_T} \ar[dr]_-{\tilde{g}_T} && R_M  \\
& R_W \ar[ru]_-{\tilde{f}_T}. }
\]
The functoriality of index morphism implies that the following commutative
\[
\xymatrix{
K^*((\gr_X)_T,( \tilde{f}_T\circ \tilde{g}_T)^*\a_T )  \ar[rr]^-{(\tilde{f}_T\circ \tilde{g}_T)_!} \ar[dr]_-{(\tilde{g}_T)_!} && K^*((\gr_M)_T, \a_T )   \\
& K^*((\gr_W)_T,( \tilde{f}_T)^*\a_T )  \ar[ru]_-{(\tilde{f}_T)_!} .}
\]
Together with the functoriality of Morita equivalences, we complete the proof
of $\tilde{f}_! \circ\tilde{g}_! =(\tilde{f} \circ\tilde{g})_!$.
\end{proof}


\begin{proposition}\label{CS47}
Let $(M/F,\sigma)$ be a twisted foliation, $g:W\longrightarrow M$ be a smooth submersion
and $F_W$ be the pull-back foliation on $W$. Denote 
$f=p_M\circ g$. Then 
\[
f_!= (p_M)! \circ g_!  : K^*(W, f^*\a +o_{F_W}) \to K^*(M/F,  \a)
\] 
with the degree shifted by the rank of $F_W$ modulo $2$.
\end{proposition}

\begin{proof} Consider the commutative diagram
\[\xymatrix{
 W \ar[d]^{g  } \ar[rr]^-{p_W  }  &&
 W/F_W \ar[d]^{\tilde{f}  }\\
 M \ar[rr]_-{ p_M  }  && M/F_M.
}
\]
By definition, $f_!=  \tilde{f}_!  \circ (p_W)_!$. 
We  shall establish the following  commutative diagram:
\begin{equation}
\xymatrix{
K^*(W,  f^*\sigma+o_{F_W})\ar[d]^{g_! }_{[\dim W-\dim M]} \ar[rr]^-{(p_W)_! }_{[rank F_W]}&&
K^*(\gr_W, \tilde{f}^*\sigma)\ar[d]^{\tilde{f}_! }\\
K^{*}(M,  \sigma_0 + o_{F} ) \ar[rr]^-{(p_M)_! } _{[rank F]}&&K^*(\gr_M,\sigma)
}
\end{equation}
where $[n]$ denotes the degree shifted by $n\mod 2$, note that  $rank F_W = \dim W-\dim M + rank F $. 

Recall  that $(p_W)_!$ and $(p_M)_!$ are defined as the composition of the Thom isomorphism  and  the analytic index morphism for $(W/F_W, \tilde f^*\a)$ and $(M/F,  \a)$ respectively, see Definition \ref{p_W}.   As the push-forward map in twisted K-theory for topological spaces is compatible with the Thom
isomorphism  (\cite{CW}), we only need to show that 
 the following diagram commutes:
\begin{equation}\label{diagB}
\xymatrix{
K^*(F^*_W,  \pi_{F_W}^* f^*\sigma)\ar[d]_-{(d_Fg)_! } \ar[rrr]^{\ind_{a, (W/F_W, \tilde f^*\a)}}&&&
K^*(\gr_W, \tilde{f}^*\sigma)\ar[d]^{\tilde{f}_! }\\
K^{*}(F^*,  \pi_{F }^*  \sigma   ) \ar[rrr]_-{\ind_{a, (M/F,  \a)}}&&&K^*(\gr_M, \sigma).
}
\end{equation}

By  the twisted index theorem  (Theorem \ref{twistlongind}), it suffices  to show   the 
 following functorial property for the twisted topological index morphisms:
\begin{equation}\label{diagB:t}
\xymatrix{
K^*(F^*_W,  \pi_{F_W}^* f^*\sigma)\ar[d]_-{(d_Fg)_! } \ar[rrr]^{\ind_{t, (W/F_W, \tilde f^*\a)}}&&&
K^*(\gr_W, \tilde{f}^*\sigma)\ar[d]^{\tilde{f}_! }\\
K^{*}(F^*,  \pi_{F }^*  \sigma   ) \ar[rrr]_-{\ind_{t, (M/F,  \a)}}&&&K^*(\gr_M, \sigma).
}
\end{equation}


 To unravel the definition of
twisted topological index (Definition \ref{twistedtopindex}), we   choose two simultaneous embeddings 
$M\stackrel{j_M}{\hookrightarrow} \RR^{2m}$ and $W\stackrel{j_W}{\hookrightarrow}\RR^{2m}$. 
Denote by $T_{M}$ and $T_W$ be open neighbourhood of  the zero section in the normal bundles to the foliations $F_M$ and $F_W$ in $\RR^{2m}$ respectively.   Using the notations in the definition of the twisted topological index and Theorem \ref{twistlongind},  we can identify 
$T_M$ and $T_W$ with open transversals of the foliations
$(M\times \RR^{2m}, \tilde{F})$ and $(M\times \RR^{2m}, \tilde{F})$ such that the following 
 diagram  commutes
 \[
\xymatrix{
T_W \ar[dd] \ar[rr]^{\subset} \ar[dr]_{\pi_{ W}}& &  U_W \ar[dr]^{\subset} \ar@{..>}[dd] &&  \\
&W \ar[dd] &  & W\times \RR^{2m} \ar[ll] \ar[dd] &\\
T_M \ar@{..>}[rr]^{\ \  \ \ \subset} \ar[dr]_{\pi_{ M}} &&   U_W \ar[dr]^{\subset} & &\\
& M  & & W\times \RR^{2m} \ar[ll] & .}
\]
Here $U_W$ and $U_M$ are small neighborhoods of $T_W$ and $T_M$ in 
$M\times \RR^{2m}$ and $ M\times \RR^{2m}$, and also identified with the corresponding normal bundles $N_W$ and $N_M$  respectively.  

Note that the twisted topological index is defined as 
the composition of  the following five homomorphisms, namely, Bott isomorphism and  Thom isomorphism in twisted K-theory of topological spaces, 
Morita isomorphism, functorial map for open embedding and Bott isomorphism in  twisted K-theory
of Lie groupoids.  We now show that  the twisted topological index  is functorial  through the following five steps.


 {\bf Step 1. } Bott isomorphism  in twisted K-theory of topological spaces is functorial. 
 
 Let $i_W:F^*_W\rightarrow F^*_W \oplus \RR^{2m} $ and $i_M :F^*_M \rightarrow F^*_M \oplus \RR^{2m} $  be  the obvious  zero sections. Then Bott isomoprhisms  are given by 
 $(i_W)_!$ and $(i_M)_!$ respectively (Cf. \cite{CW,TXL}). 
   The derivate $dg:TW\rightarrow TM$ induces $d g:F_W\rightarrow F_M$ since in this case  $F_W=Ker(q_M\circ dg)$ where $q_M:TM\rightarrow \tau_M=TM/F_M$, and $F_M=Ker\, q_M$. Letting $d g\oplus Id : F^*_W\oplus \RR^{2m}  \longrightarrow F^*_M\oplus \RR^{2m}$ be
   the K-oriented vector bundle morphism,  then the diagram
   \begin{equation}
   \label{diagI}
   \xymatrix{K^*(F_W,\pi_{F_W}^*(f^*\sigma)_0) \ar[d]_-{(d  g)_!}\ar[r]^-{Bott}&K^*(F_W\oplus \RR^m, p_{F_W}^*\pi_W^*(f^*\sigma)_0)\ar[d]^-{(d g\oplus Id )_!}\\
   K^*(F_M,\pi_{F_M}^*\pi_M^*\sigma_0)\ar[r]_-{Bott}&K^*(F_M\oplus \RR^m,\pi_M^*\sigma_0)}
   \end{equation}
    is  commutative due to  the functoriality of the pushforward maps in twisted K-theory established in  \cite{CW}.

{\bf Step 2. } Thom  isomorphism  in twisted K-theory of topological spaces is functorial. 

Denote by $T_W\stackrel{d_Tg}{\longrightarrow}T_M$ the induced vector bundle morphism which comes from  the short exact sequences of vector bundles over $W$ and $M$
 \[
 \xymatrix{
 0\ar[r]&F_W\ar[r]\ar[d]_{d g}&W\times \RR^{2m} \ar[d]_{g\times id}\ar[r]&T_W\ar[d]_{d_Tg}\ar[r]&0\\0\ar[r]&F_M\ar[r]&M\times \RR^{2m}\ar[r]&T_M\ar[r]&0
 }\]
Choose a bundle morphism  $h =d g\oplus d_Tg$ under suitable identifications 
 $W\times \RR^m\cong F_W\oplus T_W$ and $M\times \RR^m\cong  F_M\oplus T_M$.  
 Note that $F^*_W\oplus \RR^{2m}\cong  N_W\oplus N^*_W$ and $F^*_M\oplus \RR^{2m}\cong
  N_M\oplus N^*_M$ as total spaces of vector bundles over $T_W$ and $T_M$. 
 We have the following commutative diagrams of maps
  \[\xymatrix{
  F^*_W\oplus \RR^{2m} \ar[r]^{\cong}\ar[d]_{d g\oplus h}&N_W\oplus N^*_W \ar[d] &T_W\ar[l]^-{\iota}\ar[d]^{d_Tg}\\
  F^*_M\oplus \RR^{2m}\ar[r]^{\cong}& N_M\oplus N^*_M &T_M\ar[l]^-{\iota}
  }\]
  where 
  $\iota_W $ and $\iota_W $ are the  zero sections of the vector bundles $N_W\oplus N_W^* \to T_W$ and $N_M\oplus N_M^* \to T_M$.  Hence,  by  the functoriality of Thom isomorphism   in \cite{CW} and the fact that $\iota_!$ (in twisted K-theory) coincides with the Thom isomorphism in  twisted  K-theory (see again \cite{CW}), we have the following commutative diagram:\begin{equation}
  \label{diagII}
  \xymatrix{
  K^*(F^*_W\oplus \RR^{2m}, \pi_{F_W}^* (f^*\sigma)_0)\ar[r]^-{\cong}\ar[d]_-{(d g\oplus h)_!}&K^*(N_W\oplus N^*_W,\pi_N^*( \pi_W ^*f^*\a_0))  \ar[d] \ar[r]^-{Thom^{-1}}&K^*(T_W, \pi_W ^*f^*\a_0)\ar[d]^{(d_Tg)_!}
  \\
  K^*(F^*_M\oplus \RR^{2m},\pi_{F_M}^*\sigma_0)\ar[r]^-{\cong}&K^*(N_M\oplus N^*_M,\pi_N^*(\pi_M ^*\a_0)) \ar[r]_-{Thom^{-1}}&K^*(T_M,\pi_M ^*\a_0)}\end{equation}

{\bf Step 3. }  Morita isomorphism is functorial.
 In the definition of the topological indices we can choose the neighborhoods
  $U_W$ and $U_M$ small enough such that the \'etale (generalized) morphism 
  $\widetilde{f\times id }:\widetilde{\gr_W}\longrightarrow \widetilde{\gr_M}$ 
  restricts to an \'etale (generalized) morphism  $(\widetilde{f\times id})|_{U_W}:N_W\times_{T_W}N_W\longrightarrow N_M\times_{T_M}N_M$ (remember $U_W$ and $U_M$ are  chosen such that  $\widetilde{\gr}|_{U}=N\times_TN$ ). This last morphism can be equivalently described  (Definition-Proposition 1.1 in \cite{HS} or our discussion after definition \ref{p_W}) by a strict morphism of groupoids  between the \'etale groupoids obtained from the restriction  to some complete  transversals. In this case, using the complete transversals $T_W$ and $T_M$ to the foliated manifolds $(U_W,\widetilde{F_W}|_{U_W})$ and $(U_M,\widetilde{F_M}|_{U_M})$ respectively, it is easily seen that $(\widetilde{f\times id})|_{U_W}$ is given by $d_Tg:T_W\to T_M$ modulo the respective Morita equivalences $N\times_TN\stackrel{\mathscr{M}}{\longrightarrow}T$.

 By lemma \ref{tilde{f}}, the twisted K-theory morphism 
 \[
 (\widetilde{f}\times id)|_{U_W}!:K^*(\widetilde{\gr_W}|_{U_W},\widetilde{f}^*\sigma\circ p_W\circ j_W)\longrightarrow K^*(\widetilde{\gr_M}|_{U_M},\sigma\circ p_M\circ j_M)\]
  can be defined as  the morphism $(d_Tg)_!:
  K^*(T_W,\pi_W^* (f^*\a)_0)\longrightarrow K^*(T_M, \pi_M^*  \a_0)$ via the   induced morphisms by the Morita equivalences (independent of   the choice of complete transversals). In other words, the following diagram is commutative
  \begin{equation}
  \label{diagIII}
  \xymatrix{ K^*(T_W,\pi_W^* (f^*\a)_0) \ar[r]^-{Morita}_-{\cong}\ar[d]_{(d_Tg)_!}&K^*(N_W\times_{T_W}N_W, \pi_N^* \pi_W ^*(f^*\a)_0 ) \ar[r]^{\cong} & K^*(\widetilde{\gr_W}|_{U_W},\widetilde{f}^*\sigma\circ p_W\circ j_W) \ar[d]^-{((\widetilde{f\times id})|_{U_W})_!}\\ 
  K^*(T_M,\pi_M^*  \a_0)\ar[r]^-{Morita}_-{\cong}&K^*(N_M\times_{T_M}N_M,\pi_N^*  \pi_M ^*  \a _0)\ar[r]^{\cong} &  K^*(\widetilde{\gr_M}|_{U_M},\sigma\circ p_M\circ j_M), }\end{equation}

 {\bf Step 4. }  Push-forward map for open embeddings in Lie groupoids is functorial.

The morphism $\widetilde{f\times id}$ can be equivalently described by a strict morphism of groupoids 
 between the \'etale groupoids obtained from the restriction  to some complete  transversals and it does not depend on the choice of these. It is obvious that these complete transversals can be chosen such that their restriction to the $U$'s give complete transversals for the restricted foliations. Hence, by Lemma \ref{tilde{f}} and Proposition \ref{opentwist}, the following diagram is commutative 
 \begin{equation}\label{diagIV}
 \xymatrix{
 K^*(\widetilde{\gr_W}|_{U_W},\widetilde{f}^*\sigma\circ p_W\circ j_W)  \ar[r]^-{j_!}\ar[d]_-{((\widetilde{f}\times id)|_{U_W})_!}&K^*(\widetilde{\gr_W}\widetilde{f}^*\sigma\circ p_W)\ar[d]^-{(\widetilde{f}\times id)_!}\\K^*(\widetilde{\gr_M}|_{U_M}\sigma\circ p_M\circ j_M)\ar[r]_-{j_!}&K^*(\widetilde{\gr_M},\sigma\circ p_M)
 }\end{equation} 

{\bf Step 5. }  Bott isomorphism in twisted K-theory of   Lie groupoids   is functorial.
 
 For the final diagram to be commutative it is enough to observe that if $T$ is a complete transversal of $(W,F)$ then $T\times \RR^m$ is a complete transversal of $(W\times \RR^m,F\times \{0\})$. We apply Lemma \ref{tilde{f}} and Proposition \ref{Botttwist}  to obtain the commutative diagram
 \begin{equation}
 \label{diagV}
 \xymatrix{K^*(\widetilde{\gr_W},\widetilde{f}^*\sigma\circ p_W) \ar[r]^-{Bott^{-1}}\ar[d]_-{(\widetilde{f}\times id)_!}&K^*(\gr_W,\widetilde{f}^*\sigma)\ar[d]^-{\widetilde{f}_!}\\
 K^*(\widetilde{\gr_M},\sigma\circ p_M)\ar[r]_-{Bott^{-1}}&K^*(\gr_M,\sigma)}
 \end{equation} 

 Putting together the commutative diagrams (\ref{diagI}), (\ref{diagII}), (\ref{diagIII}), (\ref{diagIV}) and (\ref{diagV}), we establish  the functoriality for  the twisted topological index morphism. This completes the proof.
\end{proof}

We can now define the twisted pushforwards for smooth submersions.

\begin{definition} \label{push-submersioon}
The push-forward map $f_!: K^*(W, f^*\a + o_{TW \oplus f^*\nu_{F}}) \to K^*(M/F_M, \a  )$   for the submersion $f: W \to M/F$ is defined to  be the composition of the following maps
\begin{equation}
\xymatrix{
K^*(W, f^*\a + o_{TW \oplus f^*\nu_{F}}) \ar[r]^\cong & K^* (W, f^*\a + o_{F_W})
\ar[r]^{(p_W)_!} &K^*(W/F_W, \tilde f^*\a  ) \ar[r]^{\tilde{f}_!} & K^*(M/F_M,\a  ).
}
\end{equation}
\end{definition}
 
We want now to give a definition for any smooth generalized morphism. We will need for that purpose the following result.

 
\begin{proposition}\label{functorial}
Let $(M/F,\sigma)$ be a twisted foliation and $f: W\to M/F$ be a smooth map. Assume $f$ factors in two different ways:
\begin{equation}
\xymatrix{
&Z_1\ar[rd]^{g_1}&\\
W\ar[rd]_{j_2}\ar[ru]^{j_1}\ar[rr]^{f}&&M/F\\
&Z_2\ar[ru]_{g_2}&
}
\end{equation}
where $g_1$ and $ g_2$ are smooth submersions. Then $(g_1)_! \circ( j_1)_! =(g_2)_! \circ (j_2)_! $.
\end{proposition}

\begin{proof}
First, let us assume $j_i $ and $j_2$ are submersions. Consider the factorization of $g_1=\widetilde{g_1}\circ p_{Z_1}$. 
Putting  $h_1 = p_{Z_1}\circ j_1$,  we have a commutative diagram
\[
\xymatrix{
W\ar[rr]^{j_1}\ar[d]_{p_W}\ar[rd]_{h_1}&&Z_1\ar[d]^{g_1}\ar[dl]^{p_{Z_1}}\\
W/F_W\ar[r]_{\widetilde{h}_1}&Z_1/F_{Z_1}\ar[r]_{\widetilde{g_1}}&M/F.
}
\]
By Propositions \ref{CS47} and \ref{CS45} , we have $ (h_1)_! =  (\tilde{h}_1)_! \circ (p_W)_! =    (p_{Z_1})_! 
\circ (j_1)_! $   and  
$(\widetilde{g_1})_! \circ (\widetilde{h}_1)! =
(\widetilde{g_1} \circ  \widetilde{h}_1)_! $.
Hence, 
 \[
(g_1)_! \circ( j_1)_! = (\tilde{g}_1)_! \circ (p_{Z_1})_! 
\circ (j_1)_! =   (\tilde{g}_1)_!  \circ ( \tilde{h}_1)_! \circ (p_W)_!  = (\widetilde{g_1} \circ\widetilde{h}_1)_! 
\circ (p_W)_! .\]

Similarly, the commutative diagram
\[
\xymatrix{
W\ar[rr]^{j_1}\ar[d]_{p_W}\ar[rd]_{h_2}&&Z_2\ar[d]^{g_1}\ar[dl]^{p_{Z_2}}\\
W/F_W\ar[r]_{\widetilde{h}_2}&Z_2/F_{Z_2}\ar[r]_{\widetilde{g_1}}&M/F.
}
\]
implies that 
\[
(g_2)_! \circ( j_2)_! =    (\widetilde{g_2} \circ\widetilde{h}_2)_! 
\circ (p_W)_! .\] 
Then  $(g_1)_! \circ( j_1)_! =(g_2)_! \circ (j_2)_! $ follows $\widetilde{g_1} \circ\widetilde{h}_1 = \widetilde{g_2} \circ\widetilde{h}_2$.

For the general case, recall that Connes and Skandalis construct in \cite{CS} the following commutative diagram
\begin{equation}
\xymatrix{
&Z_1\ar[rd]^{g_1}&\\
W\ar[rd]_{j_2}\ar[ru]^{j_1}\ar[r]^{j}&Z\ar[u]_{\pi_1}\ar[d]^{\pi_2}\ar[r]^g&M/F\\
&Z_2\ar[ru]_{g_2}&
}
\end{equation}
where $Z$ is a smooth manifold, $j$ is  a smooth map and $\pi_1$, $\pi_2$ and $g$  are smooth submersions. The manifold $Z$ corresponds to the fibered product of $Z_1$ and $Z_2$ over $M/F$ (see also \cite{BH} for further discussions). The desire equality follows immediately from the wrong way functoriality in twisted K-theory for manifolds proved in \cite{CW} and the submersion case treated above.
\end{proof}

We can give the definition of the push-forward map for any smooth map $f: W\to  M/F$ where
$M/F$ is equipped with a twisting $\a$.  

\begin{definition}\label{Main:def}
Let $(M/F,\sigma)$ be a twisted foliation and $f:W\longrightarrow M/F$ be any smooth map. We define  
\[
f_!:   K^{*}(W, f^*\a_0 + o_{TW \oplus  f^*\nu_F})\longrightarrow K^*(M/F,\sigma)
\]
 to be the composition 
$g_! \circ j_! $ for any factorization $f=g\circ j$ through $g:Z\longrightarrow M/F$ a smooth submersion. Here the push-forward map \[
 j_!:  
  K^*(W, f^*\a_0 + o_{TW \oplus  f^*\nu_F}) \longrightarrow  K^*( Z,g^*\a + o_{TZ \oplus g^*\nu_F}).
  \] 
  is  established in \cite{CW}, and the push-forward map $g_!$ for the submersion $g$ is defined in
  Definition \ref{push-submersioon}, with the possible shift on degree (see also definition \ref{p_W}). 
\end{definition}

 The main result of this section can be now stated.

\begin{theorem}\label{twistfunctoriality}
The push-forward morphism is  functorial, that means, if we have a composition of smooth maps
\begin{equation}
X\stackrel{g}{\longrightarrow}W\stackrel{f}{\longrightarrow}M/F,
\end{equation}
and a twisting  $\xymatrix{\sigma:M/F \ar@{-->}[r] &PU(H)}$,  then the following diagram commutes
\[
\xymatrix{
K^{* }(X,(f\circ g)^*\sigma+ o_{TX\oplus ( f\circ g)^*\nu_F}  ) \ar[rr]^-{(f\circ g)_!}\ar[rd]_-{g_!}&&K^*(M/F,\sigma)\\
&K^{*}(W,  f^*\sigma+ o_{TW \oplus f^*\nu_F} )\ar[ru]_-{f_!}&
}
\]
\end{theorem}

\begin{proof} Choose a factorization of   $f:W\longrightarrow M/F$  as $W \stackrel{j}{\longrightarrow}Z \stackrel{q}{\longrightarrow}M/F$ where $q$ is a submersion, then 
$f\circ g = (g\circ j)  \circ q$ is a  factorization of   $f \circ g$.   Then the claim follows 
Proposition  \ref{functorial} and the functorial property of the push-forward map in \cite{CW}. \end{proof}

We remark that when $\sigma$ is trivial and $f: W\to M/F$ is K-oriented (that is  $TW\oplus \nu_F$ is K-oriented), then
 our push-forward map  in Definition \ref{Main:def}  agrees with the one constructed in \cite{CS}. Also, when the foliation consists on a single leaf (hence a manifold) but the twisting is not necessarily trivial, we obtain otherwise the push-forward map defined in \cite{CW} by Carey-Wang.

\section{Further developpements}
 
It is very natural now to use the wrong way functoriality studied in the last section to construct an assembly map adapted to our twisted situation. Indeed, it is possible to adapt to foliations the twisted geometric K-homology introduced in \cite{W08} by the second author. 
 
As indicated in \cite{CaWangCRAS}, we will obtain a twisted assembly map 
\begin{equation}
\mu_{\sigma}:K_{*}^{geo}(M/F,\sigma+o_{\nu_F})\longrightarrow K^*(M/F,\sigma). 
\end{equation} 
There are two very interesting particular cases: the first is when $\sigma$ is trivial, in this case we recover the assembly map considered by Connes-Skandalis:
$$\mu_{\tau}:K_{*}^{geo}(M/F, o_{\nu_F})\longrightarrow K^*(M/F),$$
and the second is when $\sigma=o_{\nu_F}$, in this case we obtain the following assembly map:
$$\mu:K_{*}^{geo}(M/F)\longrightarrow K^{*} (M/F, o_{\nu_F}).$$ 

The situation here is more subtle than the untwisted case since we have not developped the appropriate pseudodifferential calculus and/or we have not discussed the construction of analytic elements from (twisted) correspondences. In particular the understanding of these subjects will lead us to prove and understand the bordism invariant of our twisted indices, and then, to understand how the assembly map fits into some kind of $S^1$-equivariant assembly map (we already saw that our indices are naturally factors of a $S^1$-equivariant index). We will discuss these topics in detail in a forthcoming paper.

\bibliographystyle{plain}
\bibliography{bibliographie}

\end{document}